\documentclass[11pt,reqno]{amsart}
\usepackage{indentfirst, amssymb,amsmath,amsthm, mathrsfs,cite,graphicx,float}
\newtheorem*{theoA}{Theorem A}
\newtheorem*{theoB}{Theorem B}
\newtheorem*{theoC}{Theorem C}
\newtheorem*{theoD}{Theorem D}
\newtheorem*{theoE}{Theorem E}
\newtheorem*{theoF}{Theorem F}
\newtheorem*{theoG}{Theorem G}

\newtheorem{theo}{Theorem}[section]
\newtheorem{lem}{Lemma}[section]

\newtheorem{cor}{Corollary}[section]

\newtheorem{rem}{Remark}[section]

\newtheorem{que}{Question}[section]

\newtheorem{defi}{Definition}[section]

\newtheorem{open problem}{Open problem}[section]

\newcommand{\F}{{\mathcal{F}}}
\newcommand{\G}{{\mathcal{G}}}
\newcommand{\D}{{\Bbb{D}}}
\newcommand{\C}{{\Bbb{C}}}
\newcommand{\N}{{\Bbb{N}}}
\newcommand{\R}{{\Bbb{R}}}
\newcommand{\pa}{\partial}
\newcommand{\ol}{\overline}
\newcommand{\be}{\begin{equation}}
\newcommand{\ee}{\end{equation}}
\newcommand{\bs}{\begin{small}}
\newcommand{\es}{\end{small}}
\newcommand{\beas}{\begin{eqnarray*}}
\newcommand{\eeas}{\end{eqnarray*}}
\newcommand{\bea}{\begin{eqnarray}}
\newcommand{\eea}{\end{eqnarray}}
\renewcommand{\epsilon}{\varepsilon}
\numberwithin{equation}{section}

\begin{document}
\title[generalized Bohr inequalities] {Generalized Bohr inequalities for K-quasiconformal harmonic mappings and their applications}
\author[R. Biswas and R. Mandal]{Raju Biswas and Rajib Mandal}
\date{}
\address{Raju Biswas, Department of Mathematics, Raiganj University, Raiganj, West Bengal-733134, India.}
\email{rajubiswasjanu02@gmail.com}
\address{Rajib Mandal, Department of Mathematics, Raiganj University, Raiganj, West Bengal-733134, India.}
\email{rajibmathresearch@gmail.com}

\maketitle
\let\thefootnote\relax
\footnotetext{2020 Mathematics Subject Classification: 30A10, 30B10, 30C62, 30C75, 40A30.}
\footnotetext{Key words and phrases: Harmonic mappings, locally univalent functions,  Bohr inequality, $K$-quasiconformal mappings, Gaussian hypergeometric function.}
\footnotetext{Type set by \AmS -\LaTeX}
\begin{abstract} The classical Bohr theorem and its subsequent generalizations have become active areas of research, with investigations conducted in numerous function spaces. 
Let $\{\psi_n(r)\}_{n=0}^\infty$ be a sequence of non-negative continuous functions defined on $[0,1)$ such that the series $\sum_{n=0}^\infty \psi_n(r)$ 
converges locally uniformly on the interval $[0, 1)$. The main objective of this paper is to establish several sharp versions of generalized Bohr inequalities for the class of $K$-quasiconformal sense-preserving harmonic mappings on the unit disk $\D := \{z \in \mathbb{C} : |z| < 1\}$. To achieve these, we employ the sequence of functions $\{\psi_n(r)\}_{n=0}^\infty$ in the majorant series 
rather than the conventional dependence on the basis sequence $\{r^n\}_{n=0}^\infty$.
As applications, we derive a number of previously published results as well as a number of sharply improved and refined Bohr inequalities for harmonic mappings in $\D$. Moreover, 
we obtain a convolution counterpart of the Bohr theorem for harmonic mapping within the context of the Gaussian hypergeometric function.
\end{abstract}
\section{Introduction and Preliminaries}
In a noteworthy discovery, H. Bohr \cite{B1914} demonstrated that if $f$ is a bounded analytic function on the open unit disk $\D$ with the supremum norm 
$\Vert f\Vert_\infty:=\sup_{z\in\mathbb{D}} |f(z)|$ and the Taylor series expansion $f(z)=\sum_{n=0}^{\infty} a_nz^n$, then 
\bea\label{e2} \sum_{n=0}^{\infty}|a_n|r^n\leq \Vert f\Vert_\infty\quad\text{for}\quad|z|=r\leq\frac{1}{3}.\eea
It is observed that, if $|f(z)|\leq 1$ in $\mathbb{D}$ and $|f(z_1)|=1$ for some point $z_1\in\mathbb{D}$, then $f(z)$ reduces to a unimodular constant function by the maximum 
modulus principle. In this context, the quantity $1/3$ is known as the Bohr radius and is regarded as an optimal value that cannot be improved further. Meanwhile, the inequality 
(\ref{e2}) is known as the Bohr inequality. 
The inequality (\ref{e2}) was originally proved by Bohr \cite{B1914} for $r\leq 1/6$. Subsequently, the improved value $1/3$ was obtained independently by Weiner, Riesz, and Schur (see \cite{D1995}).
The function $f(z)=(a-z)/(1-az)$ demonstrates that $1/3$ is optimal for $a\to1^-$. 
For further insight, see the survey articles by Abu-Muhanna {\it et al.} \cite{AAP2017}, Garcia {\it et al.} (see \cite[Chapter 8]{GMR2018}), and the highly informative monograph by Defant {\it et al.} \cite{DGMS2019} on the Bohr phenomenon in the context of modern analysis of more general settings.\\[2mm]
\indent 
The concept of the Bohr phenomenon can be generalized to the class $\G$, which consists of analytic functions of the form $f(z)=\sum_{n=0}^{\infty} a_nz^n$.
These functions map from the unit disk $\D$ into a given domain $U\subseteq\C$ such that $f(\mathbb{D})\subseteq U$.
The class $\G$ is said to satisfy the Bohr phenomenon if there exists a largest number $ r_{U}\in(0, 1)$ such that the inequality (\ref{e2}) holds for $|z|= r \leq r_{U}$.
In this context, the quantity $r_U$ is referred to as the Bohr radius for the class $\G$. 
Boas and Khavinson \cite{BK1997} contributed to the revival of interest in this topic by further developing the concept of the Bohr radius in the context of several complex variables.
Furthermore, the authors have identified the multidimensional Bohr radius as a significant contribution to this field of research. 
Moreover, Dixon \cite{D1995} employed the classical Bohr theorem to develop a Banach algebra that is not categorized as an operator algebra but nevertheless satisfies the 
non-unital von Neumann inequality. Subsequently, several researchers have followed up and investigated this phenomenon in a more general abstract setting across various contexts 
(see \cite{1A2000,2A2001, BDK2004,LP2021,PS2004,1AH2022}). 
In particular, following the publication of the articles \cite{AAP2017} and \cite{2KP2018}, researchers have explored a number of approaches and novel problems pertaining to Bohr's inequality in the plane (see\cite{AAH2022,AKP2019,1AH2021,2AH2021,AA2023,BB2004,BB2004,DFOOS2011,1KP2018,LP2023,MBG2024}).
Another concept that has been the subject of considerable recent discussion is the Hankel determinant of the logarithmic coefficients of univalent functions  (see \cite{AAS2023,B2024,MRA2024}).
\subsection{Bohr-Rogosinski inequality and known important theorems} 
In addition to the notion of the Bohr radius, there is another concept known as the Rogosinski radius \cite{R1923} for bounded analytic functions in $\mathbb{D}$, which is defined as follows: Let $f(z)=\sum_{n=0}^{\infty}a_nz^n$ be analytic in $\mathbb{D}$ such that $|f(z)|<1$ in $\Bbb{D}$ and $S_N(z):=\sum_{n=0}^{N-1}a_nz^n$ denotes the partial 
sum of $f$. Then, $\left|S_N(z)\right|<1$ in the disk $|z|<1/2$ for $N\geq 1$. The number $1/2$ is optimal. 
Motivated by the Rogosinski radius, Kayumov and Ponnusamy \cite{1KP2017} have considered the Bohr-Rogosinski sum $R_N^f(z)$, which is defined by
\beas R_N^f(z):=|f(z)|+\sum_{n=N}^{\infty}|a_n||z|^n.\eeas
Clearly, $|S_N(z)|=\left|f(z)-\sum_{n=N}^{\infty}a_nz^n\right|\leq R_N^f(z)$. 
Moreover, the classical Bohr sum (majorant series) associates the Bohr-Rogosinski sum with setting the value of $N$ to 1 and substituting $|f(z)|$ with $|f(0)|$.
In their article, Kayumov and Ponnusamy \cite{1KP2017} have defined the Bohr-Rogosinski radius as the largest number $r_0\in(0,1)$ such that the inequality $R_N^f(z)\le1$ holds for the disk $|z|\leq r_0$.\\[2mm] 
\indent
In 2017, Kayumov and Ponnusamy \cite{1KP2017} obtained the following results about the Bohr-Rogosinski radius for bounded analytic functions in $\D$.
\begin{theoA}\cite{1KP2017}
Let $f(z)=\sum_{n=0}^{\infty} a_nz^n$ be analytic in $\mathbb{D}$ and $|f(z)|\leq 1$. Then
\beas |f(z)|+\sum_{n=N}^{\infty}|a_n||z|^n\leq 1\eeas
for $|z|= r\leq R_N$, where $R_N\in(0,1)$ is the positive root of the equation $\psi_N( r)=0$, $\psi_N( r)=2(1+ r) r^N-(1- r)^2$. The radius $R_N$ is the best possible. Moreover, 
\beas|f(z)|^2+\sum_{n=N}^{\infty}|a_n||z|^n\leq 1\eeas
for $|z|= r\leq R_N'$, where $R_N'\in (0,1)$ is the positive root of the equation $(1+ r) r^N-(1- r)^2=0$. The radius $R_N'$ is the best possible.
\end{theoA}
In light of the work of Kayumov and Ponnusamy \cite{1KP2017}, Liu {\it et al.} \cite{LSX2018} have investigated several results on Bohr-type inequality. Here is one of them.
\begin{theoB}\cite{LSX2018}
Let $f(z)=\sum_{n=0}^{\infty} a_nz^n$ be analytic in $\mathbb{D}$ and $|f(z)|\leq 1$. Then
\beas |f(z)|+|f'(z)||z|+\sum_{n=2}^{\infty}|a_n||z|^n\leq 1\quad\text{for}\quad |z|=r\leq (\sqrt{17}-3)/4.\eeas
The radius $(\sqrt{17}-3)/4$ is the best possible.
\end{theoB}
In addition, a number of authors have investigated other extensions of this kind (see \cite{AKP2020,LLP2021,KKP2021}). Before proceeding with the discussion, and in order to contextualize the recent results, it is essential to introduce the requisite notations. Let $h$ be an analytic function
in $\D$ and $\D_r := \{z\in\C : |z| < r\}$ for $0<r< 1$. Let $S_ r(h)$ denotes the planar integral  
\beas S_ r(h)=\int_{\mathbb{D}_ r} |h'(z)|^2 dA(z).\eeas
If $h(z)=\sum_{n=0}^\infty a_n z^n$, then it is well known that $S_ r(h)/\pi=\sum_{n=1}^\infty n|a_n|^2  r^{2n}$ and if $h$ is univalent, then $S_ r(h)$ is the area of the image $h(\mathbb{D}_ r)$. \\[2mm]
In 2018, Kayumov and Ponnusamy \cite{1KP2018} obtained the following improved versions of Bohr's inequality for the bounded analytic functions in $\mathbb{D}$.
\begin{theoC}\cite{1KP2018}
Let $f(z)=\sum_{n=0}^{\infty} a_nz^n$ be analytic in $\mathbb{D}$, $|f(z)|\leq 1$ and $S_r(f)$ denotes the area of the image of the subdisk $|z|<r$ under mapping $f$. Then
\beas \sum_{n=0}^{\infty} |a_n|r^n+\frac{16}{9}\left(\frac{S_r(f)}{\pi}\right)\leq 1\quad\text{for}\quad r\leq\frac{1}{3}.\eeas
The numbers $1/3$, $16/9$ cannot be improved. Moreover, 
\beas |a_0|^2+\sum_{n=1}^{\infty} |a_n| r^n+\frac{9}{8}\left(\frac{S_r(f)}{\pi}\right)\leq 1\quad\text{for}\quad r\leq\frac{1}{2},\eeas
and the numbers $1/2$, $9/8$ cannot be improved. \end{theoC}
\noindent 
Subsequently, Ismagilov {\it et al.}\cite{IKP2020,IKKP2021} proceeded with the investigation and yielded a number of sharply improved versions of Bohr's inequality in relation to both $S_r(f)/\pi$ and $S_r(f)/(\pi-S_r(f))$.\\[2mm]
In 2020, Ponnusamy {\it et al.} \cite{PVW2020} established the following refined Bohr inequality.
\begin{theoD}\cite{PVW2020}  Let $f(z)=\sum_{n=0}^{\infty} a_nz^n$ be analytic in $\mathbb{D}$ and $|f(z)|\leq 1$. Then,
\beas \sum_{n=0}^\infty |a_n| r^n+\left(\frac{1}{1+|a_0|}+\frac{ r}{1- r}\right)\sum_{n=1}^\infty |a_n|^2  r^{2n}\leq 1\quad\text{for}\quad  r\leq 1/(2+|a_0|),\eeas
and the numbers $1/(1+|a_0|)$ and $1/(2+|a_0|)$ cannot be improved. Moreover,
\beas |a_0|^2+\sum_{n=1}^\infty |a_n| r^n+\left(\frac{1}{1+|a_0|}+\frac{ r}{1- r}\right)\sum_{n=1}^\infty |a_n|^2  r^{2n}\leq 1\quad\text{for}\quad  r\leq 1/2,\eeas
and the numbers $1/(1+|a_0|)$ and $1/2$ cannot be improved.\end{theoD}
In their study, the authors Liu {\it et al.} \cite{LLP2021} developed several refined versions of the Bohr-Rogosinski inequality, building upon the existing Bohr-type inequalities.
Here is one of them.
\begin{theoE}\cite{LLP2021} Let $f(z)=\sum_{n=0}^{\infty} a_nz^n$ be analytic in $\mathbb{D}$ and $|f(z)|\leq 1$. Then,
\beas |f(z)|+|f'(z)|r+\sum_{n=2}^\infty |a_n| r^n+\left(\frac{1}{1+|a_0|}+\frac{r}{1- r}\right)\sum_{n=1}^\infty |a_n|^2  r^{2n}\leq 1\eeas
for $r\leq (\sqrt{17}-3)/4$. The number $(\sqrt{17}-3)/4$ is the best possible. Moreover, 
\beas |f(z)|^2+|f'(z)|r+\sum_{n=2}^\infty |a_n| r^n+\left(\frac{1}{1+|a_0|}+\frac{r}{1- r}\right)\sum_{n=1}^\infty |a_n|^2  r^{2n}\leq 1\eeas
for $r\leq r_0$, where $r_0\approx 0.385795$ is the unique positive root of the equation $1-2r-r^2-r^3-r^4=0$ and the number $r_0$ is the best possible. 
\end{theoE}
\subsection{Basic notations}
In order to present our results in an organized and coherent manner, it is necessary to introduce some basic notations at this preliminary stage of the discussion.\\[2mm]
 Let $f = u + iv$ be a complex-valued function of $z=x+i y$ in a simply connected domain $\Omega$. If $f$ satisfies the Laplace equation $\Delta f =4f_{z\ol z} = 0$  in $\Omega$, then $f$ is said to be harmonic in $\Omega$. In other words, $u$ and $v$ are real harmonic in $\Omega$. Note that every harmonic mapping $f$ has the canonical representation $f = h + \ol g$, where $h$ and 
$g$ are analytic in $\Omega$, known respectively as the analytic and co-analytic parts of $f$, and $\ol{g(z)}$ denotes the complex conjugate of $g(z)$. This representation is 
unique up to an additive constant (see \cite{D2004}). The inverse function theorem and a result of Lewy \cite{L1936} shows that a harmonic function $f$ is locally univalent in 
$\Omega$ if, and only if, the Jacobian  of $f$, defined by $J_f(z):=|h'(z)|^2-|g'(z)|^2$ is non-zero in $\Omega$. A harmonic mapping $f$ is locally univalent and sense-preserving
in $\Omega$ if, and only if, $J_f (z) > 0$ in $\Omega$ or equivalently if $h'\not=0$ in $\Omega$
and the dilatation $\omega_f:= \omega=g'/h'$ of $f$ has the property that $|\omega_f| < 1$ in $\Omega$ (see \cite{L1936}).\\[2mm]
\indent If a locally univalent and sense-preserving harmonic mapping $f = h+\ol{g}$ satisfies the condition $|g'(z)/h'(z)| \leq k < 1$ for $z\in\mathbb{D}$, then $f$ is said to be 
$K$-quasiconformal harmonic mapping on $\mathbb{D}$, where $K = (1+k)/(1-k) \geq 1$ (see \cite{K2008,M1968}). Clearly, $ k \to 1$ corresponds to the limiting case $K \to\infty$.\\[2mm]
\indent The generalization of the Bohr radius is an important study in understanding the Bohr phenomenon for various classes of functions. The idea involves replacing the basis sequence $\{r^n\}_{n=0}^\infty$ by a  more generalized sequence in the majorant series.\\[2mm]
Let $\mathcal{F}$ denote the set of all sequences $\{\psi_n(r)\}_{n=0}^\infty$ of non-negative continuous functions in $[0,1)$ such that the series $\sum_{n=0}^\infty \psi_n(r)$ 
converges locally uniformly on the interval $[0, 1)$. For an analytic function $f(z)=\sum_{n=0}^\infty a_n z^n$ in $\D$ with $|f(z)|\leq 1$, we see that 
\beas \left(\frac{1}{1+|a_0|}+\frac{r}{1-r}\right)\sum_{n=1}^\infty |a_n|^2 r^{2n}&=&\sum_{n=1}^\infty |a_n|^2\left(\frac{r^{2n}}{1+|a_0|}+\frac{r^{2n+1}}{1-r}\right)\\[2mm]
&=&\sum_{n=1}^\infty |a_n|^2\left(\frac{\psi_{2n}(r)}{1+|a_0|}+\Psi_{2n+1}(r)\right)\eeas
for $\psi_k(r)=r^k$, where $\Psi_t(r)=\sum_{k=t}^\infty \psi_k(r)$.
\subsection{A generalization of the classical Bohr theorem}
In a pioneering contribution to the literature, Kayumov {\it et al.} \cite{KKP2022} employed a different approach in their investigation of Bohr-type inequalities. Instead of relying on 
the traditional use of the basis sequence $\{r^n\}_{n=0}^\infty$, they introduced a new sequence of functions $\{\psi_n(r)\}_{n=0}^\infty\in\F$. By doing so, they were able to 
establish a generalization of the classical Bohr theorem.
\begin{theoF} \cite{KKP2022} 
Suppose that $f(z)=\sum_{n=0}^\infty a_n z^n$ is an analytic function in $\mathbb{D}$ with $|f(z)|\leq 1$ and $p\in(0,2]$. If $\{\psi_n(r)\}_{n=1}^\infty\in\F$ satisfies the inequality
\beas \frac{2}{p}\sum_{n=1}^\infty \psi_n(r)<\psi_0(r)\quad\text{for}\quad r\in[0,R),\eeas
where $R$ is minimal positive root of the equation $(2/p)\sum_{n=1}^\infty \psi_n(r)=\psi_0(r)$,
then the following inequality holds 
\beas |a_0|^p\psi_0(r)+\sum_{n=1}^\infty |a_n| \psi_n(r)\leq \psi_0(r)\quad\text{for}\quad r\leq R.\eeas
In the case when $(2/p)\sum_{n=1}^\infty \psi_n(r)>\psi_0(r)$ in some interval $(R, R+\epsilon)$, then the number $R$ cannot be improved. If the functions $\psi_k(r)$ $(k\geq 0)$ are smooth functions, then the last condition is equivalent to the inequality $(2/p)\sum_{n=1}^\infty \psi_n'(R)>\psi_0'(R)$. 
\end{theoF}
Subsequently, Ponnusamy {\it et al.} \cite{PVW2021} carried out further research in this general form and established the following refined generalization of Bohr's theorem.
\begin{theoG} \cite{PVW2021} 
Suppose that $f(z)=\sum_{n=0}^\infty a_n z^n$ is an analytic function in $\mathbb{D}$ with $|f(z)|\leq 1$ and $p\in(0,2]$. If $\{\psi_n(r)\}_{n=1}^\infty\in\F$ satisfies the inequality
\beas \frac{2}{p}\sum_{n=1}^\infty \psi_n(r)<\psi_0(r)\quad\text{for}\quad r\in[0,R),\eeas
where $R$ is minimal positive root of the equation $2\sum_{n=1}^\infty \psi_n(r)=p\psi_0(r)$,
then the following inequality holds 
\beas |a_0|^p\psi_0(r)+\sum_{n=1}^\infty |a_n| \psi_n(r)+\sum_{n=1}^\infty |a_n|^2\left(\frac{\psi_{2n}(r)}{1+|a_0|}+\Psi_{2n+1}(r)\right)\leq \psi_0(r)\quad\text{for}\quad r\leq R,\eeas
where $\Psi_t(r)=\sum_{k=t}^\infty \psi_k(r)$.
In the case when $2\sum_{n=1}^\infty \psi_n(r)>p\psi_0(r)$ in some interval $(R, R+\epsilon)$, then the number $R$ cannot be improved.
\end{theoG}
\noindent In light of the aforementioned findings, several questions naturally arise with regard to this study.
\begin{que}\label{Q1} Instead of the conventional dependence on the basis sequence $\{r^n\}_{n=0}^\infty$, is it possible to establish general results by using the sequence 
$\{\psi_n(r)\}_{n=0}^\infty\in\F$ in the majorant series for harmonic mappings in the settings of \textrm{Theorems B-E}?
 \end{que} 
\begin{que}\label{Q2} Can we establish the harmonic extension of the generalized Bohr inequality in the context of  \textrm{Theorem F}?\end{que}
\begin{que}\label{Q3} Can we establish the harmonic extension of the generalized Bohr inequality in refined form in the context of  \textrm{Theorem G}?\end{que}
\begin{que}\label{Q4} Can we establish several sharp versions of generalized Bohr inequalities for harmonic mappings in the context of \textrm{Theorems F} and \textrm{G}? \end{que}
The purpose of this paper is primarily to answer Questions \ref{Q1}-\ref{Q4} in the affirmative.\\[2mm]
The organization of the remaining part of this paper is: In Section $2$, we present a number of lemmas and a proof of one of these lemmas, which are essential for the proof of our main theorems. In Section $3$, we establish a sharp result that generalizes the Bohr inequality in an improved form for harmonic mappings. As a consequence, we establish a few 
corollaries, remarks, and several sharply improved Bohr inequalities in the context of harmonic mappings. Furthermore, we establish a convolution counterpart of the Bohr theorem for 
harmonic mapping within the context of the Gaussian hypergeometric function. In Section 4, we obtain a number of sharp results that generalize the Bohr inequality in a refined form 
for harmonic mappings. As a consequence, we establish a few 
corollaries, remarks, and several sharply refined Bohr inequalities in the context of harmonic mappings.
\section{Some lemmas}
In order to prove the main results, it is necessary to utilize the following lemmas, which are essential to this research.
\begin{lem}\label{lem1} \cite[Pick's invariant form of Schwarz’s lemma]{K2006} Suppose $f$ is analytic in $\mathbb{D}$ with $|f(z)|\leq1$, then 
\beas |f(z)|\leq \frac{|f(0)|+|z|}{1+|f(0)||z|}\quad\text{for}\quad z\in\mathbb{D}.\eeas\end{lem}
\begin{lem}\cite{DP2008}\label{lem2} Suppose $f$ is analytic in $\mathbb{D}$ with $|f(z)|\leq1$, then we have
\beas \frac{\left|f^{(n)}(z)\right|}{n!}\leq \frac{1-|f(z)|^2}{(1-|z|)^{n-1}(1-|z|^2)}\quad\text{and}\quad |a_n|\leq 1-|a_0|^2\quad\text{for}\quad n\geq 1,\;|z|<1.\eeas\end{lem}
\begin{lem}\label{lem3}\cite{C1940}\cite[Lemma B]{PVW2020} Suppose that $f\in\mathcal{B}(\mathbb{D},\mathbb{C})$ and $f(z) = \sum_{k=0}^\infty a_kz^k$. Then the following inequalities hold.
\begin{enumerate}
\item[(i)] $\vert a_{2k+1}\vert\leq 1-\vert a_0\vert^2-\cdots-\vert a_k\vert^2$, $k=0,1,\ldots$
\item[(ii)] $\vert a_{2k}\vert\leq 1-\vert a_0\vert^2-\cdots-\vert a_{k-1}\vert^2-\frac{\vert a_k\vert^2}{1+\vert a_0\vert}$, $k=1,2,\ldots$.
\end{enumerate}
Further, to have equality in $(i)$ it is necessary that $f$ is a rational function of the form 
\beas f(z)=\frac{a_0+a_1z+\cdots+a_nz^n+\epsilon z^{2n+1}}{1+(\ol{a_n}z^n+\cdots+\ol{a_0}z^{2n+1})\epsilon},\quad|\epsilon|=1,\eeas
and to have the equality in $(ii)$ it is necessary that $f$ is a rational function of the form 
\beas f(z)=\frac{a_0+a_1z+\cdots+(a_n/(1+|a_0|))z^n+\epsilon z^{2n+1}}{1+((\ol{a_n}/(1+|a_0|))z^n+\cdots+\ol{a_0}z^{2n})\epsilon},\quad|\epsilon|=1,\eeas
where the term $a_0\ol{a_n}^2\epsilon$ is non-negative real.
\end{lem}
\begin{lem}\label{lem4}\cite{PVW2021} Let $\{\psi_n(r)\}_{n=1}^\infty$ be a decreasing sequence of non-negative functions in $[0,r_\psi)$, and $h$, $g$ be analytic in $\Bbb{D}$ such that $|g'(z)|\leq k|h'(z)|$ in $\Bbb{D}$ and for some $k\in[0,1]$, where $h(z)=\sum_{n=0}^\infty a_n z^n$ and $g(z)=\sum_{n=0}^\infty b_n z^n$. Then, 
\beas \sum_{n=1}^\infty |b_n|^2 \psi_n(r)\leq k^2\sum_{n=1}^\infty |a_n|^2 \psi_n(r)\quad\text{for}\quad r\in[0,r_\psi).\eeas\end{lem}
\begin{lem}\cite[Proof of Theorem 1]{PVW2021}\label{lem6} Suppose $f(z)=\sum_{n=0}^\infty a_nz^n$ is analytic in $\mathbb{D}$ with $|f(z)|\leq1$ and $\{\psi_n(r)\}_{n=1}^\infty \in\F$. Then, we have the following inequality:
\beas \sum_{n=1}^\infty |a_n|\psi_n(r)+\sum_{n=1}^\infty |a_n|^2\left(\frac{\psi_{2n}(r)}{1+|a_0|}+\Psi_{2n+1}(r)\right)\leq \left(1-|a_0|^2\right)\Psi_1(r),\eeas
where $\Psi_t(r)=\sum_{k=t}^\infty \psi_k(r)$.
\end{lem}
In order to establish our main result, it is necessary to prove the following lemma.
\begin{lem}\label{lem7} Suppose $f(z)=\sum_{n=0}^\infty a_nz^n$ is analytic in $\mathbb{D}$ with $|f(z)|\leq1$ and $\{\psi_n(r)\}_{n=1}^\infty \in\F$. Then, we have the following inequality:
\beas \sum_{n=2}^\infty |a_n|\psi_n(r)+\sum_{n=1}^\infty |a_n|^2\left(\frac{\psi_{2n}(r)}{1+|a_0|}+\Psi_{2n+1}(r)\right)\leq \left(1-|a_0|^2\right)\Psi_2(r),\eeas
where $\Psi_t(r)=\sum_{k=t}^\infty \psi_k(r)$.
\end{lem}
\begin{proof} In view of \textrm{Lemma \ref{lem3}}, we have
\beas\sum_{n=2}^\infty |a_n|\psi_n(r)&=&\sum_{n=1}^\infty |a_{2n}|\psi_{2n}(r)+\sum_{n=1}^\infty |a_{2n+1}|\psi_{2n+1}(r)\\[2mm]
&\leq &\sum_{n=1}^\infty \left(1-|a_0|^2-\sum_{k=1}^{n-1}|a_k|^2-\frac{|a_n|^2}{1+|a_0|}\right)\psi_{2n}(r)\\[2mm]
&&+\sum_{n=1}^\infty \left(1-|a_0|^2-\sum_{k=1}^{n}|a_k|^2\right)\psi_{2n+1}(r)\\[2mm]
&=&\left(1-|a_0|^2\right)\sum_{n=2}^\infty \psi_n(r)-\sum_{n=1}^\infty\frac{|a_n|^2}{1+|a_0|}\psi_{2n}(r)\\[2mm]
&&-\left(\sum_{n=1}^\infty\left(\sum_{k=1}^{n-1}|a_k|^2\right) \psi_{2n}(r)+\sum_{n=1}^\infty\left(\sum_{k=1}^{n}|a_k|^2\right) \psi_{2n+1}(r)\right).\eeas
It is evident that 
\beas &&\sum_{n=1}^\infty\left(\sum_{k=1}^{n-1}|a_k|^2\right) \psi_{2n}(r)+\sum_{n=1}^\infty\left(\sum_{k=1}^{n}|a_k|^2\right) \psi_{2n+1}(r)\\[2mm]
&=&|a_1|^2\sum_{k=3}^\infty \psi_k(r)+|a_2|^2\sum_{k=5}^\infty \psi_k(r)+|a_3|^2\sum_{k=7}^\infty \psi_k(r)+\cdots\\[2mm]
&=&\sum_{n=1}^\infty |a_n|^2\left(\sum_{k=2n+1}^\infty \psi_k(r)\right).\eeas
Therefore, we have 
\beas \sum_{n=2}^\infty |a_n|\psi_n(r)&\leq&\left(1-|a_0|^2\right)\sum_{n=2}^\infty \psi_n(r)-\sum_{n=1}^\infty\frac{|a_n|^2}{1+|a_0|}\psi_{2n}(r)-\sum_{n=1}^\infty |a_n|^2\Psi_{2n+1}\\[2mm]
&=&\left(1-|a_0|^2\right)\sum_{n=2}^\infty \psi_n(r)-\sum_{n=1}^\infty|a_n|^2\left(\frac{\psi_{2n}(r)}{1+|a_0|}+\Psi_{2n+1}(r)\right).\eeas
This completes the proof.
\end{proof}
\section{Improved generalized Bohr inequality and its applications}
Let us consider a polynomial $G(x)$, which is defined as follows:
\bea\label{p1} G(x)=\sum_{t=1}^N c_t x^t,\quad \text{where}\quad N\in\N,\quad c_t\in\R^+\cup\{0\}\quad\text{for}\quad 1\leq t\leq N.\eea
In the following, we establish a sharp result that generalizes the Bohr inequality in an improved form for harmonic mappings.
\begin{theo}\label{T1} Suppose that $f(z)=h(z)+\ol{g(z)}=\sum_{n=0}^\infty a_n z^n+\ol{\sum_{n=1}^\infty b_n z^n}$ is a sense-preserving $K$-quasiconformal harmonic mapping in $\mathbb{D}$, where $h(z)$ is bounded in $\mathbb{D}$. If $p\in(0,2]$ and $\{\psi_n(r)\}_{n=1}^\infty\in\F$ is a decreasing sequence and satisfies the inequality
\bea\label{t1} 
\Phi_1(r):=\frac{2K}{K+1}\sum_{n=1}^\infty \psi_n(r)+\sum_{t=1}^N c_t\left(\sum_{n=1}^\infty n \left(\psi_n(r)\right)^2\right)^t-\frac{p}{2}\psi_0(r)\leq 0\eea
for $ r\leq R_{N,K}$, 
then the following inequality holds 
\beas |a_0|^p\psi_0(r)+\sum_{n=1}^\infty |a_n| \psi_n(r)+\sum_{n=1}^\infty |b_n| \psi_n(r)+G\left(\sum_{n=1}^\infty n|a_n|^2(\psi_n(r))^2\right)\leq \psi_0(r)\Vert h(z)\Vert_\infty\eeas
for $ r\leq R_{N,K}$,
where $G(x)$ is a polynomial defined in (\ref{p1}) and $R_{N,K}\in(0,1)$ is minimal positive root of the equation $\Phi_1(r)=0$.
In the case when $\Phi_1(r)>0$ in some interval $(R_{N,K}, R_{N,K}+\epsilon)$ $(\epsilon>0)$, then the number $R_{N,K}$ cannot be improved.
\end{theo}
\begin{proof}For simplicity, we assume that $\Vert h(z)\Vert_\infty\leq 1$. In view of \textrm{lemma \ref{lem2}}, we have $|a_n|\leq 1-|a_0|^2$ 
for $n\geq 1$. As $f$ is locally univalent and $K$-quasiconformal sense-preserving harmonic mapping on $\mathbb{D}$, Schwarz's Lemma gives that the dilatation $\omega=g'/h'$ 
is analytic in $\mathbb{D}$ and $|\omega(z)|=|g'(z)/h'(z)|\leq k$ in $\mathbb{D}$, where $K = (1+k)/(1-k) \geq 1$, $k\in[0,1)$. Let $|a_0|=a\in[0,1]$.
In view of \textrm{Lemma \ref{lem4}}, we have 
\bea\label{E1} \sum_{n=1}^\infty |b_n|^2 \psi_n(r)\leq k^2 \sum_{n=1}^\infty |a_n|^2 \psi_n(r)\leq k^2\left(1-a^2\right)^2\sum_{n=1}^\infty \psi_n(r).\eea
Using the Cauchy-Schwarz inequality and the inequality (\ref{E1}), we have 
\beas &&\sum_{n=1}^\infty |b_n| \psi_n(r)\leq \left(\sum_{n=1}^\infty |b_n|^2 \psi_n(r)\right)^{1/2}\left(\sum_{n=1}^\infty \psi_n(r)\right)^{1/2} \leq k(1-a^2)\sum_{n=1}^\infty \psi_n(r)\\[2mm]
&&\sum_{n=1}^\infty n|a_n|^2 \left(\psi_n(r)\right)^{2}\leq \left(1-a^2\right)^2\sum_{n=1}^\infty n \left(\psi_n(r)\right)^{2}.\eeas
Let $F(x)=(1-x^p)/(1-x^2)$, where $x\in[0,1)$ and $p\in(0,2]$. It is evident that 
\beas F'(x)=\frac{x F_1(x)}{(1-x^2)^2},\quad\text{where}\quad F_1(x)=2 - (2-p)x^p-px^{p-2}.\eeas
Therefore, $F_1'(x)=(2-p)px^{p-3} (1-x^2)\geq 0$ for $x\in[0,1)$ and $p\in(0,2]$, which shows that $F_1(x)$ is a monotonically increasing function of $x$. Thus, we have $F_1(x)\leq \lim_{x\to1^-} F_1(x)=0$ and it's shows that $F(x)$ is a monotonically decreasing function of $x$. Hence, we have $F(x)\geq \lim_{x\to1^-} F(x)=p/2$.
Therefore,
\beas &&|a_0|^p\psi_0(r)+\sum_{n=1}^\infty |a_n| \psi_n(r)+\sum_{n=1}^\infty |b_n| \psi_n(r)+G\left(\sum_{n=1}^\infty n|a_n|^2(\psi_n(r))^2\right)\\[2mm]
&\leq& a^p\psi_0(r)+(1+k)\left(1-a^2\right)\sum_{n=1}^\infty \psi_n(r)+\sum_{t=1}^N c_t\left(1-a^2\right)^{2t}\left(\sum_{n=1}^\infty n \left(\psi_n(r)\right)^2\right)^t\\[2mm]
&=&\psi_0(r)+\left(1-a^2\right)\left((1+k)\sum_{n=1}^\infty \psi_n(r)+\sum_{t=1}^N c_t \left(1-a^2\right)^{2t-1}\left(\sum_{n=1}^\infty n \left(\psi_n(r)\right)^2\right)^t\right.\\
&&\left.-\left(\frac{1-a^p}{1-a^2}\right)\psi_0(r)\right)\\[2mm]
&\leq& \psi_0(r)+\left(1-a^2\right)\left((1+k)\sum_{n=1}^\infty \psi_n(r)+\sum_{t=1}^N c_t \left(1-a^2\right)^{2t-1}\left(\sum_{n=1}^\infty n \left(\psi_n(r)\right)^2\right)^t\right.\\[2mm]
&&\left.-\frac{p}{2}\psi_0(r)\right)\\[2mm]
&=& \psi_0(r)+\left(1-a^2\right)F_1(a, r),\eeas
where $k=(K-1)/(K+1)$ and
\beas F_1(a, r)&=&(1+k)\sum_{n=1}^\infty \psi_n(r)+\sum_{t=1}^N c_t \left(1-a^2\right)^{2t-1}\left(\sum_{n=1}^\infty n \left(\psi_n(r)\right)^2\right)^t-\frac{p}{2}\psi_0(r).\eeas
Differentiating partially $F_1(a, r)$ with respect to $a$, we obtain
\beas \frac{\pa}{\pa a} F_1(a, r)=-2a\sum_{t=1}^N c_t (2t-1)\left(1-a^2\right)^{2t-2}\left(\sum_{n=1}^\infty n \left(\psi_n(r)\right)^2\right)^t\leq 0.\eeas 
Therefore, $F_1(a, r)$ is a monotonically decreasing function of $a\in[0,1]$ and it follows that 
\beas F_1(a, r)\leq F_1(0,r)=\frac{2K}{K-1}\sum_{n=1}^\infty \psi_n(r)+\sum_{t=1}^N c_t\left(\sum_{n=1}^\infty n \left(\psi_n(r)\right)^2\right)^t-\frac{p}{2}\psi_0(r)=\Phi_1(r)\leq 0\eeas
for $r\leq R_{N,K}$, under the assumption that (\ref{t1}) holds.\\[2mm]
\indent To prove the sharpness of the result, we consider the function $f_1(z)=h_1(z)+\ol{g_1(z)}$ in $\mathbb{D}$ such that 
\beas h_1(z)=\frac{a-z}{1-az}=A_0+\sum_{n=1}^\infty A_n z^n,\eeas
where $A_0=a$, $A_n=-(1-a^2)a^{n-1}$ for $n\geq 1$, $a\in[0,1)$ and $g_1(z)=\lambda k \sum_{n=1}^\infty A_n z^n$, where $|\lambda|=1$ and $k=(K-1)/(K+1)$.
Thus,
\beas &&|A_0|^p\psi_0(r)+\sum_{n=1}^\infty |A_n| \psi_n(r)+\sum_{n=1}^\infty |k\lambda A_n| \psi_n(r)+G\left(\sum_{n=1}^\infty n|A_n|^2(\psi_n(r))^2\right)\\[2mm]
&=&\psi_0(r)+\left(1-a^2\right)(1+k)\sum_{n=1}^\infty a^{n-1} \psi_n(r)\\
&&+\sum_{t=1}^N c_t \left(1-a^2\right)^{2t}\left(\sum_{n=1}^\infty n a^{2(n-1)} \left(\psi_n(r)\right)^2\right)^t-(1-a^p)\psi_0(r)\\[2mm]
&=& \psi_0(r)+\left(1-a^2\right)F_2(a,r)+\sum_{t=1}^N c_t \left(1-a^2\right)^{2t}\left(\sum_{n=1}^\infty n a^{2(n-1)}\left(\psi_n(r)\right)^2\right)^t\\[2mm]
&&-\left(1-a^2\right)\sum_{t=1}^N c_t \left(\sum_{n=1}^\infty n a^{2(n-1)}\left(\psi_n(r)\right)^2\right)^t,\eeas
where $k=(K-1)/(K+1)$ and
\beas F_2(a,r)=(1+k)\sum_{n=1}^\infty \psi_n(r)+\sum_{k=1}^N c_k\left(\sum_{n=1}^\infty n a^{2(n-1)}\left(\psi_n(r)\right)^2\right)^k-\frac{1-a^p}{1-a^2}\psi_0(r).\eeas
Letting $a\to 1^-$, we have
\beas&& |A_0|^p\psi_0(r)+\sum_{n=1}^\infty |A_n| \psi_n(r)+\sum_{n=1}^\infty |k\lambda A_n| \psi_n(r)+G\left(\sum_{n=1}^\infty n|A_n|^2(\psi_n(r))^2\right)\\[2mm]
&=&\psi_0(r)+\left(1-a^2\right)\Psi_1(r)+O\left(\left(1-a^2\right)^2\right).\eeas 
Since $\Phi_1(r)>0$ for $(R_{N,K}, R_{N,K}+\epsilon)$, thus the radius cannot be improved. This completes the proof.
\end{proof}
\begin{rem}Setting $K=1$ and $G(x)\equiv 0$ in \textrm{Theorem \ref{T1}} gives \textrm{Theorem F}.\end{rem}
\noindent The following result represents the harmonic extension of the generalized Bohr inequality in the context of \textrm{Theorem F}, by setting $G(x)\equiv 0$ in \textrm{Theorem \ref{T1}}.  
\begin{cor}\label{C1}
Suppose that $f(z)=h(z)+\ol{g(z)}=\sum_{n=0}^\infty a_n z^n+\ol{\sum_{n=1}^\infty b_n z^n}$ is a sense-preserving $K$-quasiconformal harmonic mapping in $\mathbb{D}$, where $h(z)$ is bounded in $\mathbb{D}$. If $p\in(0,2]$ and $\{\psi_n(r)\}_{n=1}^\infty\in\F$ is a decreasing sequence and satisfies the inequality
\beas
\Phi_2(r):=\frac{2K}{K+1}\sum_{n=1}^\infty \psi_n(r)-\frac{p}{2}\psi_0(r)\leq 0\eeas
for $ r\leq R_{K}$, 
then the following inequality holds 
\beas |a_0|^p\psi_0(r)+\sum_{n=1}^\infty |a_n| \psi_n(r)+\sum_{n=1}^\infty |b_n| \psi_n(r)\leq \psi_0(r)\Vert h(z)\Vert_\infty\eeas
for $ r\leq R_{K}$,
where $R_{K}\in(0,1)$ is minimal positive root of the equation $\Phi_2(r)=0$.
In the case when $\Phi_2(r)>0$ in some interval $(R_{K}, R_{K}+\epsilon)$ $(\epsilon>0)$, then the number $R_{K}$ cannot be improved.
\end{cor}
\subsection{Applications of \textrm{Theorem \ref{T1}}}
The following results are counterparts of Bohr's theorem in different settings, derived from the applications of \textrm{Theorem \ref{T1}}.
\begin{itemize}
\item[(1)] Let $\psi_n(r)=r^n$ for $n\geq 0$, $p=1$ and $G(x)\equiv 0$ in \textrm{Theorem \ref{T1}}. Then, we obtain \textrm{Theorem 1} of \cite{KPS2018}, {\it i.e.,} 
\beas \sum_{n=0}^\infty |a_n|r^n+\sum_{n=1}^\infty |b_n| r^n\leq \Vert h\Vert_\infty\quad\text{for}\quad r\leq \frac{K+1}{5K+1}.\eeas
The number $(K+1)/(5K+1)$ is sharp.
\item[(2)] Let $\psi_n(r)=r^n$ for $n\geq 0$, $p=2$ and $G(x)\equiv 0$ in \textrm{Theorem \ref{T1}}. Then, we obtain \textrm{Theorem 2} of  \cite{KPS2018}, {\it i.e.,} 
\beas |a_0|^2+\sum_{n=1}^\infty |a_n|r^n+\sum_{n=1}^\infty |b_n| r^n\leq \Vert h\Vert_\infty\quad\text{for}\quad r\leq \frac{K+1}{3K+1}.\eeas
The number $(K+1)/(3K+1)$ is sharp.
\item[(3)] Let $\psi_0(r)=1$, $\psi_n(r)=r^n/n$ for $n\geq 1$, $p=1$ and $G(x)=x$ in \textrm{Theorem \ref{T1}}.  As a consequence, we obtain the following sharp Bohr inequality.
\beas |a_0|+\sum_{n=1}^\infty \frac{|a_n|}{n}r^n+\sum_{n=1}^\infty \frac{|b_n|}{n} r^n+\sum_{n=1}^\infty \frac{|a_n|^2}{n} r^{2n}\leq \Vert h\Vert_\infty\quad\text{for}\quad r\leq r_0,\eeas
where $r_0\in(0,1)$ is the unique positive root of the equation
\beas \frac{4K}{K+1}\log(1-r)+2\log\left(1-r^2\right)+1=0.\eeas
The number $r_0$ is the best possible.
\item[(4)] Let $\psi_0(r)=1$, $\psi_n(r)=r^n/n^2$ for $n\geq 1$, $p=1$ and $G(x)=x$ in \textrm{Theorem \ref{T1}}. As a consequence, we obtain the following sharp Bohr inequality.
\beas |a_0|+\sum_{n=1}^\infty \frac{|a_n|}{n^2}r^n+\sum_{n=1}^\infty \frac{|b_n|}{n^2} r^n+\sum_{n=1}^\infty \frac{|a_n|^2}{n^3} r^{2n}\leq \Vert h\Vert_\infty\quad\text{for}\quad r\leq r_0,\eeas
where $r_0\in(0,1)$ is the unique positive root of the equation
\beas \frac{4K}{K+1}\text{\it Li}_2(r)+2\text{\it Li}_3(r^2)-1=0,\eeas
where $\text{\it Li}_2(r)$ and $\text{\it Li}_3(r^2)$ are dilogarithm and trilogarithm, respectively. The number $r_0$ is the best possible.\\[1mm]
Note that the polylogarithm is a special function, which is defined as $\text{\it Li}_s(z)=\sum_{n=1}^\infty z^n/n^s$. The special case $s=1$ involves the ordinary natural logarithm, $\text{\it Li}_1(z)=-\log(1-z)$.
\end{itemize}
\begin{itemize}
\item[(5)] Let $\psi_n(r)=r^n$ for $n\geq 0$, $p=1$ and $G(x)\equiv c_1 x$ $(c_1\in\R^+)$ in \textrm{Theorem \ref{T1}}. 
As a consequence, we obtain the following sharply improved version of Bohr's inequality within the context of \textrm{Theorem 1} in \cite{KPS2018}.
\end{itemize}
\begin{cor} Suppose that $f(z)=h(z)+\ol{g(z)}=\sum_{n=0}^\infty a_n z^n+\ol{\sum_{n=1}^\infty b_n z^n}$ is a sense-preserving $K$-quasiconformal harmonic mapping in $\mathbb{D}$, where $h(z)$ is bounded in $\mathbb{D}$. Then,
\beas \sum_{n=0}^\infty |a_n|r^n+\sum_{n=1}^\infty |b_n| r^n+c_1\left(\frac{S_r(h)}{\pi}\right)\leq \Vert h\Vert_\infty\quad\text{for}\quad r\leq R_{1,K},\eeas
where $R_{1,K}\in(0,1)$ is the unique positive root of the equation 
\beas F_3(r):=\frac{2K}{K+1}\left(\frac{r}{1-r}\right)+c_1\frac{r^2}{\left(1-r^2\right)^2}-\frac{1}{2}=0.\eeas
The number $R_{1,K}$ is the best possible.
\end{cor}
\begin{rem}Note that, the function $F_3(r)$ is a monotonically increasing function of $r\in[0,1]$ with $F_3(0)=-1/2<0$ and $\lim_{r\to1^-}F_3(r)=+\infty$. Therefore, $R_{1,K}$ is the unique positive root of the equation in $(0,1)$.\end{rem}
\begin{itemize}
\item[(6)] Let $\psi_n(r)=r^n$ for $n\geq 0$, $p=2$ and $G(x)\equiv c_1 x$ $(c_1\in\R^+)$ in \textrm{Theorem \ref{T1}}. 
As a result, we obtain the following sharply improved version of Bohr's inequality within the framework of \textrm{Theorem 2} in \cite{KPS2018}.
\end{itemize}
\begin{cor} Suppose that $f(z)=h(z)+\ol{g(z)}=\sum_{n=0}^\infty a_n z^n+\ol{\sum_{n=1}^\infty b_n z^n}$ is a sense-preserving $K$-quasiconformal harmonic mapping in $\mathbb{D}$, where $h(z)$ is bounded in $\mathbb{D}$. Then,
\beas |a_0|^2+\sum_{n=1}^\infty |a_n|r^n+\sum_{n=1}^\infty |b_n| r^n+c_1\left(\frac{S_r(h)}{\pi}\right)\leq \Vert h\Vert_\infty\quad\text{for}\quad r\leq R_{2,K},\eeas
where $R_{2,K}\in(0,1)$ is the unique positive root of the equation 
\beas \frac{2K}{K+1}\left(\frac{r}{1-r}\right)+c_1\frac{r^2}{\left(1-r^2\right)^2}-1=0.\eeas
The number $R_{2,K}$ is the best possible.
\end{cor}
\begin{itemize}
\item[(7)] Let $\psi_n(r)=r^n$ for $n\geq 0$, $p=1$ and $G(x)\equiv \sum_{t=1}^N c_t x^t$ $(c_t\in\R^+,\; N\in\N)$ in \textrm{Theorem \ref{T1}}. 
As a result, we obtain the following sharply improved version of Bohr's inequality within the context of \textrm{Theorem 1} in \cite{KPS2018}.
\end{itemize}
\begin{cor}Suppose that $f(z)=h(z)+\ol{g(z)}=\sum_{n=0}^\infty a_n z^n+\ol{\sum_{n=1}^\infty b_n z^n}$ is a sense-preserving $K$-quasiconformal harmonic mapping in $\mathbb{D}$, where $h(z)$ is bounded in $\mathbb{D}$. Then, we have
\beas \sum_{n=0}^\infty |a_n|r^n+\sum_{n=1}^\infty |b_n| r^n+c_1\left(\frac{S_r(h)}{\pi}\right)+c_2\left(\frac{S_r(h)}{\pi}\right)^2+\ldots+c_N\left(\frac{S_r(h)}{\pi}\right)^N\leq \Vert h\Vert_\infty\eeas
for $r\leq R_{N,K}$, where $R_{N,K}\in(0,1)$ is the unique positive root of the equation 
\beas F_4(r):=\frac{2K}{K+1}\left(\frac{r}{1-r}\right)+\sum_{t=1}^N c_t\left(\frac{r^2}{\left(1-r^2\right)^2}\right)^t-\frac{1}{2}=0.\eeas
The number $R_{N,K}$ is the best possible.
\end{cor}
\begin{rem} It is evident that $F_4(r)$ is a monotonically increasing function of $r\in[0,1]$ with $F_4(0)=-1/2<0$ and $\lim_{r\to1^-}F_4(r)=+\infty$. \end{rem}
\begin{itemize}
\item[(8)] Let $\psi_n(r)=r^n$ for $n\geq 0$, $p=2$ and $G(x)\equiv \sum_{t=1}^N c_t x^t$ $(c_t\in\R^+,\; N\in\N)$ in \textrm{Theorem \ref{T1}}. Then, we obtain the following sharply improved version of Bohr inequality in the context of \textrm{Theorem 2} in  \cite{KPS2018}.\end{itemize}
\begin{cor}Suppose that $f(z)=h(z)+\ol{g(z)}=\sum_{n=0}^\infty a_n z^n+\ol{\sum_{n=1}^\infty b_n z^n}$ is a sense-preserving $K$-quasiconformal harmonic mapping in $\mathbb{D}$, where $h(z)$ is bounded in $\mathbb{D}$. Then, we have
\beas |a_0|^2+\sum_{n=1}^\infty\left( |a_n|+ |b_n|\right) r^n+\sum_{t=1}^N c_t\left(\frac{S_r(h)}{\pi}\right)^t\leq \Vert h\Vert_\infty\eeas
for $r\leq R$, where $R\in(0,1)$ is the unique positive root of the equation 
\beas \frac{2K}{K+1}\left(\frac{r}{1-r}\right)+\sum_{t=1}^N c_t\left(\frac{r^2}{\left(1-r^2\right)^2}\right)^t-1=0.\eeas\end{cor}
\subsection{Convolution counterpart of Bohr theorem for harmonic mappings}
To prove our convolution result, we need the following definitions.
\begin{defi}\cite{SS1989} Let $\psi_1$ and $\psi_2$ be two analytic functions in $\mathbb{D}$ given by $\psi_1(z)=\sum_{n=0}^{\infty}a_nz^n$ and $\psi_2(z)=\sum_{n=0}^{\infty}b_nz^n$. The convolution (or, Hadamard product) is defined by 
\beas\left(\psi_1*\psi_2\right)(z)=\sum_{n=0}^{\infty}a_nb_nz^n\quad\text{for}\quad z\in\mathbb{D}.\eeas \end{defi}
It is evident that $\psi_1*\psi_2=\psi_2*\psi_1$. In 2002, Goodloe \cite{G2002} considered the Hadamard product of a harmonic function with an analytic function as follows:
\beas f\tilde{*}\phi=h*\phi+\ol{g*\phi},\eeas
where $f=h+\ol{g}$ is a harmonic mapping in $\D$, and $\phi$ is an analytic function in $\mathbb{D}$.\\[2mm]
\indent  The Gaussian hypergeometric function is a highly significant special function due to its extensive connections with other classes of special functions as specific or limiting cases and its numerous identities and expressions in terms of series and integrals. It is a solution of Euler's hypergeometric differential equation
\beas  z(1-z)y''+\left(c-(a+b+1)z\right)y'-aby=0,\eeas 
which has three regular singular points at $0$, $1$ and $\infty$. 
\begin{defi}
For parameters $a, b, c\in\C$ with $c\not=-m$ $(m\in\N\cup\{0\})$ and $z\in\D$, the Gaussian hypergeometric function is defined by
\beas F(a, b, c, z):={}_2F_1(a,b;c;z)=\sum_{n=0}^\infty \gamma_n z^n,\quad\text{where}\quad \gamma_n=\frac{(a)_n(b)_n}{(c)_n(1)_n},\eeas
which converges for all $z\in\D$ and converges on the circle $|z|=1$ if $\text{Re}(c-a-b)>0$. Here $(a)_n$ is the Pochhammer symbol, which is defined by:
\beas
(a)_n:=\left\{\begin{array}{lll}
1,&\text{if}\quad n=0\\
a(a+1)(a+2)\cdots(a+n-1),&\text{if}\quad n>0.\end{array}\right.
\eeas
\end{defi}
It is evident that the Gaussian hypergeometric function $F(a,b;c;z)$ is analytic within the domain $\D$, and it can also be analytically continued outside the unit circle.
The function reduces to a polynomial when either $a$ or $b$ is a non-positive integer, {\it i.e.,}
\beas {}_2F_1(-m,b;c;z)=\sum_{n=0}^m (-1)^n \binom{m}{n}\frac{(b)_n}{(c)_n} z^n.\eeas
Some of the simplest special cases of the hypergeometric functions are
\beas {}_2F_1(a,1;1;z)=(1-z)^{-a}, {}_2F_1(1,1;2;z)=-\log(1-z)/z.\eeas
For further details, we refer to \cite{QV2005} and the references cited therein.\\[2mm] 
\indent We consider the convolution operator of a harmonic function $f(z)=h(z)+\ol{g(z)}=\sum_{n=0}^\infty a_n z^n+\ol{\sum_{n=1}^\infty b_n z^n}$ with an analytic function $F(a, b, c, z)=\sum_{n=0}^\infty \gamma_n z^n$ as follows:
\beas \left(f\tilde{*}F\right)(z)=\sum_{n=0}^\infty \gamma_n a_n z^n+\ol{\sum_{n=1}^\infty \gamma_n b_n z^n}.\eeas 
As a consequence of \textrm{Theorem \ref{T1}}, we obtain the following sharp Bohr radius for harmonic mappings within the context of the Gaussian hypergeometric function. 
\begin{theo} Suppose that $f(z)=h(z)+\ol{g(z)}=\sum_{n=0}^\infty a_n z^n+\ol{\sum_{n=1}^\infty b_n z^n}$ is a sense-preserving $K$-quasiconformal harmonic mapping in 
$\mathbb{D}$, where $h(z)$ is bounded in $\mathbb{D}$ and $p\in(0,2]$. Let $a, b, c\in\R^+$ be such that $(a+n)(b+n)r-(c+n)(1+n)\leq0$ for $n\geq 0$ and $r\in[0,1]$. Then,
\beas |a_0|^p+\sum_{n=1}^\infty \gamma_n|a_n|r^n+\sum_{n=1}^\infty \gamma_n|b_n|r^n\leq \Vert h\Vert_\infty \quad\text{for}\quad r\leq r_0, \eeas
where $r_0\in(0,1)$ is the minimal positive root of the equation $F(a, b, c, r)-1=p(K+1)/(4K)$. The number $r_0$ cannot be improved.
\end{theo}
\begin{proof}Let $\psi_n(r)=\gamma_n r^n$ for $n\geq 0$. It is evident that $\gamma_0=1$ and
\beas F(a, b, c, r)-1=\sum_{n=1}^\infty \gamma_n r^n=\sum_{n=1}^\infty\psi_n.\eeas
Note that $(a)_{n+1}=a(a+1)(a+2)\cdots(a+n)=(a+n)(a)_n$ for $n\in\N$. 
Clearly, $\gamma_n=(a)_n(b)_n/(c)_n(1)_n>0$ for $n\geq 1$, since $a,b,c\in\R^+$.
Therefore, we have
\beas &&\gamma_{n+1}=\frac{(a+n)(b+n)}{(c+n)(1+n)}\gamma_n\quad\text{and}\\[2mm]
&&\psi_{n+1}(r)-\psi_n(r)=\gamma_{n+1} r^{n+1}-\gamma_n r^n=\gamma_n r^n\left(\frac{(a+n)(b+n)}{(c+n)(1+n)}r-1\right)\leq 0\eeas
for $n\geq 0$, by the assumption. Thus, $\{\psi_n(r)\}_{n=1}^\infty\in\F$ is a decreasing sequence. In view of \textrm{Corollary \ref{C1}}, we have
\beas |a_0|^p+\sum_{n=1}^\infty |\gamma_n||a_n|r^n+\sum_{n=1}^\infty |\gamma_n||b_n|r^n\leq \Vert h\Vert_\infty\eeas
 holds for $r\leq r_0$, where $r_0$ is the minimal positive root of the equation $F(a, b, c, r)-1=p(K+1)/(4K)$.
\end{proof}
\section{Refined generalized Bohr-type inequalities and their applications}
In the following, we establish a sharp result that generalizes the Bohr inequality in a refined form for harmonic mappings.
\begin{theo}\label{T2} Suppose that $f(z)=h(z)+\ol{g(z)}=\sum_{n=0}^\infty a_n z^n+\ol{\sum_{n=1}^\infty b_n z^n}$ is a sense-preserving $K$-quasiconformal harmonic mapping in $\mathbb{D}$, where $h(z)$ is bounded in $\mathbb{D}$. If $p\in(0,2]$ and $\{\psi_n(r)\}_{n=1}^\infty\in\F$ is a decreasing sequence and satisfies the inequality
\bea\label{t2} 
\Phi_3(r):=\frac{2K}{K+1}\sum_{n=1}^\infty \psi_n(r)+\sum_{t=1}^N c_t\left(\sum_{n=1}^\infty n \left(\psi_n(r)\right)^2\right)^t-\frac{p}{2}\psi_0(r)\leq 0\eea
for $ r\leq R_{N,K}$, 
then the following inequality holds 
\beas&& |a_0|^p\psi_0(r)+\sum_{n=1}^\infty |a_n| \psi_n(r)+\sum_{n=1}^\infty |b_n| \psi_n(r)+\sum_{n=1}^\infty |a_n|^2\left(\frac{\psi_{2n}(r)}{1+|a_0|}+\Psi_{2n+1}(r)\right)\\[2mm]
&&+G\left(\sum_{n=1}^\infty n|a_n|^2(\psi_n(r))^2\right)\leq \psi_0(r)\Vert h(z)\Vert_\infty\eeas
for $ r\leq R_{N,K}$,
where $\Psi_t(r)=\sum_{k=t}^\infty \psi_k(r)$, $G(x)$ is defined in (\ref{p1}) and $R_{N,K}\in(0,1)$ is minimal positive root of the equation $\Phi_3(r)=0$.
In the case when $\Phi_3(r)>0$ in some interval $(R_{N,K}, R_{N,K}+\epsilon)$ $(\epsilon>0)$, then the number $R_{N,K}$ cannot be improved.
\end{theo}
\begin{proof}For simplicity, we assume that $\Vert h(z)\Vert_\infty\leq 1$. In view of \textrm{lemma \ref{lem2}}, we have $|a_n|\leq 1-|a_0|^2$ 
for $n\geq 1$. 
Using arguments as in the proof of \textrm{Theorem \ref{T1}} and in view of \textrm{Lemmas \ref{lem2}}, \ref{lem4} with the condition $|g'(z)|\leq k|h'(z)|$ in $\D$, we have 
\bea\label{E2} \sum_{n=1}^\infty |b_n|^2 \psi_n(r)\leq k^2 \sum_{n=1}^\infty |a_n|^2 \psi_n(r)\leq k^2\left(1-a^2\right)^2\sum_{n=1}^\infty \psi_n(r),\eea
 where $k= (K-1)/(K+1)$, $k\in[0,1)$ and $|a_0|=a\in[0,1)$.
In view of Cauchy-Schwarz inequality and the inequality (\ref{E2}), we have 
\bea\label{E3}\sum_{n=1}^\infty |b_n| \psi_n(r)\leq \left(\sum_{n=1}^\infty |b_n|^2 \psi_n(r)\right)^{1/2}\left(\sum_{n=1}^\infty \psi_n(r)\right)^{1/2} \leq k(1-a^2)\sum_{n=1}^\infty \psi_n(r)\\[2mm]\text{and}\quad
\label{E4}\sum_{n=1}^\infty n|a_n|^2 \left(\psi_n(r)\right)^{2}\leq \left(1-a^2\right)^2\sum_{n=1}^\infty n \left(\psi_n(r)\right)^{2}.\hspace{4cm}\eea
By employing the conclusion of \textrm{Lemma \ref{lem7}}, we have 
\bea\label{E5} \sum_{n=1}^\infty |a_n|\psi_n(r)+\sum_{n=1}^\infty |a_n|^2\left(\frac{\psi_{2n}(r)}{1+|a_0|}+\Psi_{2n+1}(r)\right)\leq \left(1-a^2\right)\Psi_1(r),\eea
where $\Psi_t(r)=\sum_{k=t}^\infty \psi_k(r)$.
From (\ref{E3}), (\ref{E4}) and (\ref{E5}), we have 
\beas &&|a_0|^p\psi_0(r)+\sum_{n=1}^\infty |a_n| \psi_n(r)+\sum_{n=1}^\infty |b_n| \psi_n(r)+\sum_{n=1}^\infty |a_n|^2\left(\frac{\psi_{2n}(r)}{1+|a_0|}+\Psi_{2n+1}(r)\right)\\
&&+G\left(\sum_{n=1}^\infty n|a_n|^2(\psi_n(r))^2\right)\\[2mm]
&\leq& a^p\psi_0(r)+(1+k)\left(1-a^2\right)\sum_{n=1}^\infty \psi_n(r)+\sum_{t=1}^N c_t\left(1-a^2\right)^{2t}\left(\sum_{n=1}^\infty n \left(\psi_n(r)\right)^2\right)^t.\eeas
The rest of the calculations are derived from \textrm{Theorem \ref{T1}}.\\[2mm]
\indent To prove the sharpness of the result, we consider the function $f_3(z)=h_3(z)+\ol{g_3(z)}$ in $\mathbb{D}$ such that 
\beas h_3(z)=\frac{a-z}{1-az}=A_0+\sum_{n=1}^\infty A_n z^n,\eeas
where $A_0=a$, $A_n=-(1-a^2)a^{n-1}$ for $n\geq 1$, $a\in[0,1)$ and $g_3(z)=\lambda k \sum_{n=1}^\infty A_n z^n$, where $|\lambda|=1$ and $k=(K-1)/(K+1)$.
Thus, 
\beas &&|A_0|^p\psi_0(r)+\sum_{n=1}^\infty |A_n| \psi_n(r)+\sum_{n=1}^\infty |k\lambda A_n| \psi_n(r)\\
&&+\sum_{n=1}^\infty |A_n|^2\left(\frac{\psi_{2n}(r)}{1+|A_0|}+\Psi_{2n+1}(r)\right)+G\left(\sum_{n=1}^\infty n|A_n|^2(\psi_n(r))^2\right)\\[2mm]
&=& a^p\psi_0(r)+\left(1-a^2\right)(1+k)\sum_{n=1}^\infty a^{n-1} \psi_n(r)+G\left(\left(1-a^2\right)^2\sum_{n=1}^\infty n a^{2n-2}(\psi_n(r))^2\right)\\
&&+\left(1-a^2\right)^2\sum_{n=1}^\infty a^{2n-2}\left(\frac{\psi_{2n}(r)}{1+a}+\Psi_{2n+1}(r)\right)\\[2mm]
&=& \psi_0(r)+\left(1-a^2\right)F_3(a,r)+\sum_{t=1}^N c_t \left(1-a^2\right)^{2t}\left(\sum_{n=1}^\infty n a^{2n-2}\left(\psi_n(r)\right)^2\right)^t\\[2mm]
&&+\left(1-a^2\right)^2\sum_{n=1}^\infty a^{2n-2}\left(\frac{\psi_{2n}(r)}{1+a}+\Psi_{2n+1}(r)\right)\\
&&-\left(1-a^2\right)\sum_{t=1}^N c_t \left(\sum_{n=1}^\infty n a^{2n-2}\left(\psi_n(r)\right)^2\right)^t,\eeas
where $k=(K-1)/(K+1)$ and
\beas F_3(a,r)=(1+k)\sum_{n=1}^\infty \psi_n(r)+\sum_{k=1}^N c_k\left(\sum_{n=1}^\infty n a^{2n-2}\left(\psi_n(r)\right)^2\right)^k-\frac{1-a^p}{1-a^2}\psi_0(r).\eeas
Letting $a\to 1^-$, we have
\beas&& |A_0|^p\psi_0(r)+\sum_{n=1}^\infty |A_n| \psi_n(r)+\sum_{n=1}^\infty |k\lambda A_n| \psi_n(r)+\sum_{n=1}^\infty |A_n|^2\left(\frac{\psi_{2n}(r)}{1+|A_0|}+\Psi_{2n+1}(r)\right)\\
&&+G\left(\sum_{n=1}^\infty n|A_n|^2(\psi_n(r))^2\right)\\[2mm]
&=&\psi_0(r)+\left(1-a^2\right)\Phi_3(r)+O\left(\left(1-a^2\right)^2\right).\eeas 
Since $\Phi_3(r)>0$ for $(R_{N,K}, R_{N,K}+\epsilon)$, thus the radius cannot be improved. This completes the proof.
\end{proof}
\begin{rem}Setting $K=1$ and $G(x)\equiv 0$ in \textrm{Theorem \ref{T2}} gives \textrm{Theorem G}.\end{rem}
\noindent The following result presents the harmonic extension of the generalized Bohr inequality in a refined form within the context of \textrm{Theorem G}, whereby the condition $G(x)\equiv 0$ is applied in \textrm{Theorem \ref{T2}}. 
\begin{cor}\label{C2}
Suppose that $f(z)=h(z)+\ol{g(z)}=\sum_{n=0}^\infty a_n z^n+\ol{\sum_{n=1}^\infty b_n z^n}$ is a sense-preserving $K$-quasiconformal harmonic mapping in $\mathbb{D}$, where $h(z)$ is bounded in $\mathbb{D}$. If $p\in(0,2]$ and $\{\psi_n(r)\}_{n=1}^\infty\in\F$ is a decreasing sequence and satisfies the inequality
\beas
\Phi_4(r):=\frac{2K}{K+1}\sum_{n=1}^\infty \psi_n(r)-\frac{p}{2}\psi_0(r)\leq 0\eeas
for $ r\leq R_{K}$, 
then the following inequality holds 
\beas |a_0|^p\psi_0(r)+\sum_{n=1}^\infty (|a_n|+ |b_n|) \psi_n(r)+\sum_{n=1}^\infty |a_n|^2\left(\frac{\psi_{2n}(r)}{1+|a_0|}+\Psi_{2n+1}(r)\right)\leq \psi_0(r)\Vert h(z)\Vert_\infty\eeas
for $ r\leq R_{K}$,
where $R_{K}$ is minimal positive root of the equation $\Phi_4(r)=0$.
In the case when $\Phi_4(r)>0$ in some interval $(R_{K}, R_{K}+\epsilon)$ $(\epsilon>0)$, then the number $R_{K}$ cannot be improved.
\end{cor}
\subsection{Applications of \textrm{Theorem \ref{T2}}}
The following results are counterparts of Bohr's theorem in different settings, derived from the applications of \textrm{Theorem \ref{T2}}.
\begin{itemize}
\item[(1)] Let $\psi_n(r)=r^n$ for $n\geq 0$, $p=1$ and $G(x)\equiv 0$ in \textrm{Theorem \ref{T2}}. Then, we obtain the sharp refined Bohr inequality of \textrm{Theorem 1} of  \cite{KPS2018} without compromising the radius, {\it i.e.,} 
\beas \sum_{n=0}^\infty |a_n|r^n+\sum_{n=1}^\infty |b_n| r^n+\left(\frac{1}{1+|a_0|}+\frac{r}{1-r}\right)\sum_{n=1}^\infty |a_n|^2 r^{2n}\leq \Vert h\Vert_\infty\eeas
for $r\leq (K+1)/(5K+1)$. The number $(K+1)/(5K+1)$ is sharp.
\item[(2)] Let $\psi_n(r)=r^n$ for $n\geq 0$, $p=2$ and $G(x)\equiv 0$ in \textrm{Theorem \ref{T2}}. Then, we obtain \textrm{Theorem 2} of  \cite{KPS2018}, {\it i.e.,} 
\beas |a_0|^2+\sum_{n=1}^\infty |a_n|r^n+\sum_{n=1}^\infty |b_n| r^n+\left(\frac{1}{1+|a_0|}+\frac{r}{1-r}\right)\sum_{n=1}^\infty |a_n|^2 r^{2n}\leq \Vert h\Vert_\infty\eeas
for $r\leq (K+1)/(3K+1)$. The number $(K+1)/(3K+1)$ is sharp.
\item[(3)] Let $\psi_0(r)=1$, $\psi_n(r)=r^n/n$ for $n\geq 1$, $p=1$ and $G(x)=x$ in \textrm{Theorem \ref{T2}}. As a consequence, we obtain the following sharp Bohr inequality.
\beas |a_0|+\sum_{n=1}^\infty \frac{\left(|a_n|+|b_n|\right)}{n}r^n+\sum_{n=1}^\infty \frac{|a_n|^2}{n} r^{2n}+\left(\frac{1}{1+|a_0|}+\frac{r}{1-r}\right)\sum_{n=1}^\infty \frac{|a_n|^2}{n^2} r^{2n}\leq \Vert h\Vert_\infty,\eeas
for $r\leq r_0$, where $r_0$ is the unique positive root of the equation
\beas \frac{4K}{K+1}\log(1-r)+2\log\left(1-r^2\right)+1=0.\eeas
The number $r_0$ is the best possible.
\item[(4)] Let $\psi_0(r)=1$, $\psi_n(r)=r^n/n^2$ for $n\geq 1$, $p=2$ and $G(x)=x$ in \textrm{Theorem \ref{T2}}. As a consequence, we obtain the following sharp Bohr inequality.
\beas |a_0|^2+\sum_{n=1}^\infty \frac{\left(|a_n|+|b_n|\right)}{n^2}r^n+\left(\frac{1}{1+|a_0|}+\frac{r}{1-r}\right)\sum_{n=1}^\infty \frac{|a_n|^2}{n^4} r^{2n}+\sum_{n=1}^\infty \frac{|a_n|^2}{n^3} r^{2n}\leq \Vert h\Vert_\infty\eeas
for $r\leq r_0$, where $r_0$ is the unique positive root of the equation
\beas \frac{2K}{K+1}\text{\it Li}_2(r)+\text{\it Li}_3(r^2)+1=0,\eeas
where $\text{\it Li}_2(r)$ and $\text{\it Li}_3(r^2)$ are dilogarithm and trilogarithm, respectively. The number $r_0$ is the best possible.
\end{itemize}
In the following, we derive a sharp result that generalizes the Bohr inequality in a refined form within the context of harmonic mappings.
\begin{theo}\label{T3} Suppose that $f(z)=h(z)+\ol{g(z)}=\sum_{n=0}^\infty a_n z^n+\ol{\sum_{n=1}^\infty b_n z^n}$ is a sense-preserving $K$-quasiconformal harmonic mapping in $\mathbb{D}$, where $h(z)$ is bounded in $\mathbb{D}$. If $p\in(0,2]$ and $\{\psi_n(r)\}_{n=1}^\infty\in\F$ is a decreasing sequence with $\psi_0(r)\not=0$ and satisfies the inequality
\be\label{t3} 
\Phi_5(r):=2\psi_1(r)+2(1+r)^2\left(\frac{2K}{K+1}\sum_{n=2}^\infty \psi_n(r)+\left(\frac{K-1}{K+1}\right)\psi_1(r)\right)-(1-r^2)\psi_0(r)\leq 0\ee
for $ r\leq R_{K}$, then the following inequality holds 
\beas &&|h(z)|\psi_0(r)+|h'(z)|\psi_1(r)+\sum_{n=2}^\infty |a_n| \psi_n(r)+\sum_{n=1}^\infty |b_n| \psi_n(r)\\[2mm]
&&+\sum_{n=1}^\infty |a_n|^2\left(\frac{\psi_{2n}(r)}{1+|a_0|}+\Psi_{2n+1}(r)\right)\leq \psi_0(r)\Vert h(z)\Vert_\infty\eeas
for $ r\leq R_{K}\leq R$, where $R\in(0,1)$ is the minimal positive root of the equation $2\psi_1(r)=\left(1-r^2\right)\psi_0(r)$
 and $R_{K}$ is minimal positive root of the equation $\Phi_5(r)=0$.
In the case when $\Phi_5(r)>0$ in some interval $(R_{K}, R_{K}+\epsilon)$ $(\epsilon>0)$, then the number $R_{K}$ cannot be improved.
\end{theo}
\begin{proof}
For simplicity, we assume that $\Vert h(z)\Vert_\infty\leq 1$. In view of \textrm{lemma \ref{lem2}}, we have $|a_n|\leq 1-|a_0|^2$ 
for $n\geq 1$. 
Using arguments as in the proof of \textrm{Theorem \ref{T1}} and in view of \textrm{Lemmas \ref{lem2}}, \ref{lem4} with the condition $|g'(z)|\leq k|h'(z)|$ in $\D$, we have the inequalities (\ref{E3}) and (\ref{E4}),
where $k= (K-1)/(K+1)$, $k\in[0,1)$ and $|a_0|=a\in[0,1]$. In view of \textrm{Lemma \ref{lem7}}, we have 
\be\label{E8} \sum_{n=2}^\infty |a_n|\psi_n(r)+\sum_{n=1}^\infty |a_n|^2\left(\frac{\psi_{2n}(r)}{1+|a_0|}+\Psi_{2n+1}(r)\right)\leq \left(1-a^2\right)\Psi_2(r).\ee
Let $Q(x)=x+\beta(1-x^2)$, where $0\leq x\leq \alpha(\leq1)$ and $\beta\geq 0$. Therefore, $Q'(x)=1-2\beta x$, 
$Q''(x)=-2\beta\leq 0$. It follows that $Q'(x)$ is a monotonically decreasing function of $x$, and hence, we have $Q'(x)\geq Q'(\alpha)=1-2\beta\alpha\geq 0$ for $\beta\leq 1/(2\alpha)$. 
Therefore, we have $Q(x)\leq Q(\alpha)$ for $0\leq \beta\leq 1/(2\alpha)$. 
In consideration of \textrm{Lemmas \ref{lem1}} and \ref{lem2}, and utilizing the inequalities (\ref{E3}), (\ref{E4}), and (\ref{E8}), we have 
\beas &&|h(z)|\psi_0(r)+|h'(z)|\psi_1(r)+\sum_{n=2}^\infty |a_n| \psi_n(r)+\sum_{n=1}^\infty |b_n| \psi_n(r)\nonumber\\[2mm]
&&+\sum_{n=1}^\infty |a_n|^2\left(\frac{\psi_{2n}(r)}{1+|a_0|}+\Psi_{2n+1}(r)\right)\eeas
\bea&\leq&\psi_0(r)\left(|h(z)|+\frac{\psi_1(r)}{\left(1-r^2\right)\psi_0(r)}\left(1-|h(z)|^2\right)\right)+(1+k)\left(1-a^2\right)\sum_{n=2}^\infty \psi_n(r)\nonumber\\[2mm]
&&+k\left(1-a^2\right)\psi_1(r)\nonumber\\[2mm]
\label{E0}&\leq&\psi_0(r)\left(\frac{a+r}{1+ar}+\frac{\psi_1(r)}{\left(1-r^2\right)\psi_0(r)}\left(1-\left(\frac{a+r}{1+ar}\right)^2\right)\right)\\[2mm]
&&+(1+k)\left(1-a^2\right)\sum_{n=2}^\infty \psi_n(r)+k\left(1-a^2\right)\psi_1(r)\nonumber\\[2mm]
&=&\psi_0(r)+\frac{(1-a)F_4(a,r)}{1+ar},\nonumber\eea
where 
\beas F_4(a,r)=\frac{(1+a)\psi_1(r)}{1+ar}+(1+ar)\left(1+a\right)\left((1+k)\sum_{n=2}^\infty \psi_n(r)+k\psi_1(r)\right)-(1-r)\psi_0(r)\eeas
and the inequality (\ref{E0}) holds for $\psi_1(r)/\left(\left(1-r^2\right)\psi_0(r)\right)\leq 1/2$, {\it i.e.,} for $r\in[0,R]$, where $R\in(0,1)$ is the minimal positive root of the equation $2\psi_1(r)-\left(1-r^2\right)\psi_0(r)=0$.
Differentiating partially $F_4(a,r)$ with respect to $a$, we obtain 
\beas\frac{\pa}{\pa a}F_4(a,r)=\left(1+r+2ar\right)\left((1+k)\sum_{n=2}^\infty \psi_n(r)+k\psi_1(r)\right)+\frac{(1-r)\psi_1(r)}{(1+ar)^2}\geq 0.\eeas
Therefore, $F_4(a,r)$ is a monotonically increasing function of $a\in[0, 1]$ and it follows that $F_4(a,r)\leq F_4(1,r)$ and
\beas F_4(1,r)=\frac{1}{1+r}\left(2\psi_1(r)+2(1+r)^2\left((1+k)\sum_{n=2}^\infty \psi_n(r)+k\psi_1(r)\right)-\left(1-r^2\right)\psi_0(r)\right),\eeas
which is less than or equal to zero for $r\leq R_K$, under the assumption that (\ref{t3}) holds. We now assert that $R_K\leq R$. For $r>R$, we have $2\psi_1(r)>\left(1-r^2\right)\psi_0$ and
\beas\Phi_5(r)>2(1+r)^2\left(\frac{2K}{K+1}\sum_{n=2}^\infty \psi_n(r)+\left(\frac{K-1}{K+1}\right)\psi_1(r)\right)\geq0,\eeas
which shows that $R_K\leq R$.\\[2mm]
\indent To prove the sharpness of the result, we consider the function $f_4(z)=h_4(z)+\ol{g_4(z)}$ in $\mathbb{D}$ such that 
\beas h_4(z)=\frac{a-z}{1-az}=A_0+\sum_{n=1}^\infty A_n z^n,\eeas
where $A_0=a$, $A_n=-(1-a^2)a^{n-1}$ for $n\geq 1$, $a\in[0,1)$ and $g_4(z)=\lambda k \sum_{n=1}^\infty A_n z^n$, where $|\lambda|=1$ and $k=(K-1)/(K+1)$.
Also $h_4'(z)=-\left(1-a^2\right)/(1-az)^2$.
Thus, 
\beas &&|h_4(-r)|\psi_0(r)+|h_4'(-r)|\psi_1(r)+\sum_{n=2}^\infty |A_n| \psi_n(r)+\sum_{n=1}^\infty |\lambda k A_n| \psi_n(r)\\[2mm]
&&+\sum_{n=1}^\infty |A_n|^2\left(\frac{\psi_{2n}(r)}{1+|A_0|}+\Psi_{2n+1}(r)\right)\\[2mm]
&=&\left(\frac{a+r}{1+ar}\right)\psi_0(r)+\left(\frac{1-a^2}{(1+ar)^2}\right)\psi_1(r)+\left(1-a^2\right)(1+k)\sum_{n=2}^\infty a^{n-1} \psi_n(r)\\[2mm]
&&+k\left(1-a^2\right)\psi_1(r)+\left(1-a^2\right)^2\sum_{n=1}^\infty a^{2n-2}\left(\frac{\psi_{2n}(r)}{1+a}+\Psi_{2n+1}(r)\right)\\[2mm]
&=&\psi_0(r)+\frac{(1-a)F_5(a,r)}{(1+ar)^2}+\left(1-a^2\right)^2\sum_{n=1}^\infty a^{2n-2}\left(\frac{\psi_{2n}(r)}{1+a}+\Psi_{2n+1}(r)\right),\eeas
where $k=(K-1)/(K+1)$ and 
\beas F_5(a,r)&=&\left(1+a\right)\psi_1(r)+(1+ar)^2\left(1+a\right)\left((1+k)\sum_{n=2}^\infty \psi_n(r)+k\psi_1(r)\right)\\[2mm]
&&-(1-r)(1+ar)\psi_0(r).\eeas
Letting $a\to 1^-$, we have
\beas&&|h_4(-r)|\psi_0(r)+|h_4'(-r)|\psi_1(r)+\sum_{n=2}^\infty |A_n| \psi_n(r)+\sum_{n=1}^\infty |B_n| \psi_n(r)\\[2mm]
&&+\sum_{n=1}^\infty |A_n|^2\left(\frac{\psi_{2n}(r)}{1+|A_0|}+\Psi_{2n+1}(r)\right)\hspace{10cm}\\[2mm]
&=&\psi_0(r)+\frac{(1-a)}{(1+ar)^2}\Phi_5(r)+O\left(\left(1-a\right)^2\right).\eeas 
Since $\Phi_5(r)>0$ for $(R_{K}, R_{K}+\epsilon)$, thus the radius cannot be improved. This completes the proof.
\end{proof}
\subsection{Applications of \textrm{Theorem \ref{T3}}}
The following results are counterparts of Bohr's theorem in different settings, derived from the applications of \textrm{Theorem \ref{T3}}.
\begin{itemize}
\item[(1)] Let $\psi_n(r)=r^n$ for $n\geq 0$ in \textrm{Theorem \ref{T3}}. As a consequence, we obtain the sharp refined Bohr-type inequality of \textrm{Theorem 1} of  \cite{KPS2018}, {\it i.e.,} 
\beas |h(z)|+|h'(z)|r+\sum_{n=2}^\infty |a_n| r^n+\sum_{n=1}^\infty |b_n| r^n+\left(\frac{1}{1+|a_0|}+\frac{r}{1-r}\right)\sum_{n=1}^\infty |a_n|^2 r^{2n}\leq \Vert h\Vert_\infty\eeas
for $r\leq R_1\leq \sqrt{2}-1$, where $R_1$ is the unique positive root of the equation
\beas (1-r)\left(r^2+2r-1\right)+2(1+r)^2\left(r^2+\left(\frac{K-1}{K+1}\right)r\right)=0.\eeas 
The number $R_1$ is the best possible.
\item[(2)] Let $\psi_0(r)=1$ and $\psi_n(r)=r^n/n$ for $n\geq 1$ in \textrm{Theorem \ref{T3}}. As a result, we obtain the following sharp Bohr-type inequality.
\beas |h(z)|+|h'(z)|r+\sum_{n=2}^\infty \frac{|a_n|}{n}r^n+\sum_{n=1}^\infty \frac{|b_n|}{n}r^n+\left(\frac{1}{1+|a_0|}+\frac{r}{1-r}\right)\sum_{n=1}^\infty \frac{|a_n|^2}{n^2} r^{2n}\leq \Vert h\Vert_\infty,\eeas
for $r\leq R_2\leq \sqrt{2}-1$, where $R_2$ is the unique positive root of the equation
\beas\left(r^2+2r-1\right)-2(1+r)^2\left(r+\frac{2K}{K+1}\log(1-r)\right)=0.\eeas 
The number $R_2$ is the best possible.
\item[(3)] Let $\psi_0(r)=1$ and $\psi_n(r)=r^n/n^2$ for $n\geq 1$ in \textrm{Theorem \ref{T3}}. As a consequence, we obtain the following sharp Bohr-type inequality.
\beas |h(z)|+|h'(z)|r+\sum_{n=2}^\infty \frac{|a_n|}{n^2}r^n+\sum_{n=1}^\infty \frac{|b_n|}{n^2}r^n+\left(\frac{1}{1+|a_0|}+\frac{r}{1-r}\right)\sum_{n=1}^\infty \frac{|a_n|^2}{n^4} r^{2n}\leq \Vert h\Vert_\infty\eeas
for $r\leq R_3\leq \sqrt{2}-1$, where $R_3$ is the unique positive root of the equation
\beas\left(r^2+2r-1\right)+2(1+r)^2\left(\frac{2K}{K+1}\text{\it Li}_2(r)-r\right)=0,\eeas 
where $\text{\it Li}_2(r)$ is a dilogarithm. The number $R_3$ is the best possible.
\end{itemize}
In the following, we derive a sharp result, which is a generalization of the Bohr-type inequality in refined form within the context of harmonic mappings.
\begin{theo}\label{T4} Suppose that $f(z)=h(z)+\ol{g(z)}=\sum_{n=0}^\infty a_n z^n+\ol{\sum_{n=1}^\infty b_n z^n}$ is a sense-preserving $K$-quasiconformal harmonic mapping in $\mathbb{D}$, where $h(z)$ is bounded in $\mathbb{D}$. If $p\in(0,2]$ and $\{\psi_n(r)\}_{n=1}^\infty\in\F$ is a decreasing sequence with $\psi_0(r)\not=0$ and satisfies the inequality
\be\label{t4} 
\Phi_6(r):=\psi_1(r)+(1+r)^2\left(\frac{2K}{K+1}\sum_{n=2}^\infty \psi_n(r)+\left(\frac{K-1}{K+1}\right)\psi_1(r)\right)-\left(1-r^2\right)\psi_0(r)\leq 0\ee
for $ r\leq R_{K}$, then the following inequality holds 
\beas &&|h(z)|^2\psi_0(r)+|h'(z)|\psi_1(r)+\sum_{n=2}^\infty |a_n| \psi_n(r)+\sum_{n=1}^\infty |b_n| \psi_n(r)\\[2mm]
&&+\sum_{n=1}^\infty |a_n|^2\left(\frac{\psi_{2n}(r)}{1+|a_0|}+\Psi_{2n+1}(r)\right)\leq \psi_0(r)\Vert h(z)\Vert_\infty\eeas
for $ r\leq R_{K}\leq R$, where $R\in(0,1)$ is the minimal positive root of the equation $2\psi_1(r)=\left(1-r^2\right)\psi_0(r)$
 and $R_{K}$ is minimal positive root of the equation $\Phi_6(r)=0$.
In the case when $\Phi_6(r)>0$ in some interval $(R_{K}, R_{K}+\epsilon)$ $(\epsilon>0)$, then the number $R_{K}$ cannot be improved.
\end{theo}
\begin{proof}
For simplicity, we assume that $\Vert h(z)\Vert_\infty\leq 1$. In view of \textrm{lemma \ref{lem2}}, we have $|a_n|\leq 1-|a_0|^2$ 
for $n\geq 1$. Let $|a_0|=a\in[0,1]$.
Using arguments as in the proof of \textrm{Theorem \ref{T3}} and in view of \textrm{Lemmas \ref{lem2}}, \ref{lem4} and \ref{lem7} with the condition $|g'(z)|\leq k|h'(z)|$ in 
$\D$, we have the inequalities (\ref{E3}), (\ref{E4}), (\ref{E8}), where $k= (K-1)/(K+1)$, $k\in[0,1)$.
In consideration of \textrm{Lemmas \ref{lem1}} and \ref{lem2}, and utilizing the inequalities (\ref{E3}), (\ref{E4}), and (\ref{E8}), we have 
\bea &&|h(z)|^2\psi_0(r)+|h'(z)|\psi_1(r)+\sum_{n=2}^\infty |a_n| \psi_n(r)+\sum_{n=1}^\infty |b_n| \psi_n(r)\nonumber\\[2mm]
&&+\sum_{n=1}^\infty |a_n|^2\left(\frac{\psi_{2n}(r)}{1+|a_0|}+\Psi_{2n+1}(r)\right)\nonumber\\[2mm]
&\leq&\psi_0(r)\left(|h(z)|^2+\frac{\psi_1(r)}{\left(1-r^2\right)\psi_0(r)}\left(1-|h(z)|^2\right)\right)+(1+k)\left(1-a^2\right)\sum_{n=2}^\infty \psi_n(r)\nonumber\\[2mm]
&&+k\left(1-a^2\right)\psi_1(r)\nonumber\\[2mm]
\label{F0}&\leq&\psi_0(r)\left(\left(1-\frac{\psi_1(r)}{\left(1-r^2\right)\psi_0(r)}\right)\left(\frac{a+r}{1+ar}\right)^2+\frac{\psi_1(r)}{\left(1-r^2\right)\psi_0(r)}\right)\\[2mm]
&&+(1+k)\left(1-a^2\right)\sum_{n=2}^\infty \psi_n(r)+k\left(1-a^2\right)\psi_1(r)\nonumber\\[2mm]
&=&\psi_0(r)-\left(1-a^2\right)F_6(a,r),\nonumber\eea
where 
\beas F_6(a,r)=\left(\psi_0(r)-\frac{\psi_1(r)}{\left(1-r^2\right)}\right)\frac{\left(1-r^2\right)}{(1+ar)^2}-(1+k)\sum_{n=2}^\infty \psi_n(r)-k\psi_1(r).\eeas
and the inequality (\ref{F0}) holds for $\psi_1(r)/\left(\left(1-r^2\right)\psi_0(r)\right)\leq 1$, {\it i.e.,} for $r\in[0,R]$, where $R\in(0,1)$ is the minimal positive root of the equation $\psi_1(r)-\left(1-r^2\right)\psi_0(r)=0$.
Differentiating partially $F_6(a,r)$ with respect to $a$, we obtain 
\beas\frac{\pa}{\pa a}F_6(a,r)=-\left(\psi_0(r)-\frac{\psi_1(r)}{\left(1-r^2\right)}\right)\frac{2r\left(1-r^2\right)}{(1+ar)^3}\leq 0.\eeas
Therefore, $F_6(a,r)$ is a monotonically decreasing function of $a\in[0, 1]$ and it follows that $F_6(a,r)\geq F_6(1,r)$ and
\beas F_6(1,r)=\frac{1}{(1+r)^2}\left(\left(1-r^2\right)\psi_0(r)-(1+r)^2\left((1+k)\sum_{n=2}^\infty \psi_n(r)+k\psi_1(r)\right)-\psi_1(r)\right),\eeas
which is greater than or equal to zero for $r\leq R_K$, under the assumption that (\ref{t4}) holds. We now assert that $R_K\leq R$. For $r>R$, we have $\psi_1(r)>\left(1-r^2\right)\psi_0$ and
\beas \Phi_6(r)>(1+r)^2\left(\frac{2K}{K+1}\sum_{n=2}^\infty \psi_n(r)+\left(\frac{K-1}{K+1}\right)\psi_1(r)\right)\geq 0,\eeas
which shows that $R_K\leq R$.\\[2mm]
\indent To prove the sharpness of the result, we consider the function $f_5(z)=h_5(z)+\ol{g_5(z)}$ in $\mathbb{D}$ such that 
\beas h_5(z)=\frac{a-z}{1-az}=A_0+\sum_{n=1}^\infty A_n z^n,\eeas
where $A_0=a$, $A_n=-(1-a^2)a^{n-1}$ for $n\geq 1$, $a\in[0,1)$ and $g_4(z)=\lambda k \sum_{n=1}^\infty A_n z^n$, where $|\lambda|=1$ and $k=(K-1)/(K+1)$.
Also $h_5'(z)=-\left(1-a^2\right)/(1-az)^2$.
Thus, 
\beas &&|h_4(-r)|^2\psi_0(r)+|h_4'(-r)|\psi_1(r)+\sum_{n=2}^\infty |A_n| \psi_n(r)+\sum_{n=1}^\infty |\lambda k A_n| \psi_n(r)\\[2mm]
&&+\sum_{n=1}^\infty |A_n|^2\left(\frac{\psi_{2n}(r)}{1+|A_0|}+\Psi_{2n+1}(r)\right)\\[2mm]
&=&\left(\frac{a+r}{1+ar}\right)^2\psi_0(r)+\frac{\left(1-a^2\right)\psi_1(r)}{(1+ar)^2}+\left(1-a^2\right)(1+k)\sum_{n=2}^\infty a^{n-1} \psi_n(r)\\[2mm]
&&+k\left(1-a^2\right)\psi_1(r)+\left(1-a^2\right)^2\sum_{n=1}^\infty a^{2n-2}\left(\frac{\psi_{2n}(r)}{1+a}+\Psi_{2n+1}(r)\right)\\[2mm]
&=&\psi_0(r)+\frac{\left(1-a^2\right)F_7(a,r)}{(1+ar)^2}+\left(1-a^2\right)^2\sum_{n=1}^\infty a^{2n-2}\left(\frac{\psi_{2n}(r)}{1+a}+\Psi_{2n+1}(r)\right),\eeas
where $k=(K-1)/(K+1)$ and 
\beas F_7(a,r)=\psi_1(r)+(1+ar)^2\left((1+k)\sum_{n=2}^\infty \psi_n(r)+k\psi_1(r)\right)-\left(1-r^2\right)\psi_0(r).\eeas
Letting $a\to 1^-$, we have
\beas&&|h_4(-r)|^2\psi_0(r)+|h_4'(-r)|\psi_1(r)+\sum_{n=2}^\infty |A_n| \psi_n(r)+\sum_{n=1}^\infty |B_n| \psi_n(r)\\[2mm]
&&+\sum_{n=1}^\infty |A_n|^2\left(\frac{\psi_{2n}(r)}{1+|A_0|}+\Psi_{2n+1}(r)\right)\hspace{10cm}\\[2mm]
&=&\psi_0(r)+\frac{\left(1-a^2\right)}{(1+ar)^2}\Phi_6(r)+O\left(\left(1-a^2\right)^2\right).\eeas 
Since $\Phi_6(r)>0$ for $(R_{K}, R_{K}+\epsilon)$, thus the radius cannot be improved. This completes the proof.
\end{proof}
\subsection{Applications of \textrm{Theorem \ref{T4}}}
The following results are counterparts of Bohr's theorem in different settings, derived from the applications of \textrm{Theorem \ref{T4}}.
\begin{itemize}
\item[(1)] Let $\psi_n(r)=r^n$ for $n\geq 0$ in \textrm{Theorem \ref{T4}}. As a consequence, we obtain the sharp refined Bohr-type inequality of \textrm{Theorem 1} of  \cite{KPS2018}, {\it i.e.,} 
\beas |h(z)|^2+|h'(z)|r+\sum_{n=2}^\infty |a_n| r^n+\sum_{n=1}^\infty |b_n| r^n+\left(\frac{1}{1+|a_0|}+\frac{r}{1-r}\right)\sum_{n=1}^\infty |a_n|^2 r^{2n}\leq \Vert h\Vert_\infty\eeas
for $r\leq R_1\leq (\sqrt{5}-1)/2$, where $R_1$ is the unique positive root of the equation
\beas (1-r)\left(r^2+r-1\right)+(1+r)^2\left(r^2+\left(\frac{K-1}{K+1}\right)r\right)=0.\eeas 
The number $R_1$ is the best possible.
\item[(2)] Let $\psi_0(r)=1$ and $\psi_n(r)=r^n/n$ for $n\geq 1$ in \textrm{Theorem \ref{T4}}. As a consequence, we obtain the following sharp Bohr-type inequality.
\beas |h(z)|^2+|h'(z)|r+\sum_{n=2}^\infty \frac{|a_n|}{n}r^n+\sum_{n=1}^\infty \frac{|b_n|}{n}r^n+\left(\frac{1}{1+|a_0|}+\frac{r}{1-r}\right)\sum_{n=1}^\infty \frac{|a_n|^2}{n^2} r^{2n}\leq \Vert h\Vert_\infty,\eeas
for $r\leq R_2\leq (\sqrt{5}-1)/2$, where $R_2$ is the unique positive root of the equation
\beas\left(r^2+r-1\right)-(1+r)^2\left(r+\frac{2K}{K+1}\log(1-r)\right)=0.\eeas 
The number $R_2$ is the best possible.
\item[(3)] Let $\psi_0(r)=1$ and $\psi_n(r)=r^n/n^2$ for $n\geq 1$ in \textrm{Theorem \ref{T4}}. As a result, we obtain the following sharp Bohr-type inequality.
\beas |h(z)|^2+|h'(z)|r+\sum_{n=2}^\infty \frac{|a_n|}{n^2}r^n+\sum_{n=1}^\infty \frac{|b_n|}{n^2}r^n+\left(\frac{1}{1+|a_0|}+\frac{r}{1-r}\right)\sum_{n=1}^\infty \frac{|a_n|^2}{n^4} r^{2n}\leq \Vert h\Vert_\infty\eeas
for $r\leq R_3\leq (\sqrt{5}-1)/2$, where $R_3$ is the unique positive root of the equation
\beas\left(r^2+r-1\right)+(1+r)^2\left(\frac{2K}{K+1}\text{\it Li}_2(r)-r\right)=0,\eeas 
where $\text{\it Li}_2(r)$ is a dilogarithm. The number $R_3$ is the best possible.
\end{itemize}
\section{Declarations}
\noindent{\bf Acknowledgment:} The work of the first author is supported by University Grants Commission (IN) fellowship (No. F. 44 - 1/2018 (SA - III)). The authors like to
thank the anonymous reviewers and and the editing team for their valuable suggestions towards the improvement of the paper.\\[2mm]
{\bf Conflict of Interest:} The authors declare that there are no conflicts of interest regarding the publication of this paper.\\[1mm]
{\bf Availability of data and materials:} Not applicable.\\[1mm]


\begin{thebibliography}{33}
\bibitem{AAP2017} {\sc Y. Abu-Muhanna, R. M. Ali and S. Ponnusamy}, On the Bohr inequality,  In: N. K. Govil {\it et al.} (eds) Progress in approximation theory and applicable complex analysis, Springer Optimization and Its Applications, {\bf117} (2017), 269--300. 
\bibitem{1A2000} {\sc L. Aizenberg}, Multidimensional analogues of Bohr's theorem on power series, {\it Proc. Amer. Math. Soc.} {\bf128} (2000), 1147--1155.
\bibitem{2A2001} {\sc L. Aizenberg, A. Aytuna and P. Djakov}, Generalization of theorem on Bohr for bases in spaces of holomorphic functions of several complex variables, {\it J. Math. Anal. Appl.} {\bf258} (2001), 429--447.
\bibitem{AKP2019} {\sc S. A. Alkhaleefah, I. R. Kayumov and S. Ponnusamy}, On the Bohr inequality with a fixed zero coefficient, {\it Proc. Am. Math. Soc.} {\bf147}(12) (2019), 5263--5274.
\bibitem{AKP2020} {\sc S. A. Alkhaleefah, I. R. Kayumov and S. Ponnusamy}, Bohr-Rogosinski inequalities for bounded analytic functions, {\it Lobachevskii J. Math.} {\bf41} (2020), 2110--2119.
\bibitem{AAH2022} {\sc M. B. Ahamed, V. Allu and H. Halder}, The Bohr phenomenon for analytic functions on a shifted disk, {\it Ann. Fenn. Math.} {\bf47} (2022), 103--120.
\bibitem{AAS2023}{\sc V. Allu, V. Arora and A. Shaji}, On the second Hankel determinant of logarithmic coefficients for certain univalent functions, {\it Mediterr. J. Math.} {\bf20} (2023), 81.
\bibitem{1AH2021} {\sc V. Allu and H. Halder}, Bohr radius for certain classes of starlike and convex univalent functions, {\it J. Math. Anal. Appl.} {\bf493}(1) (2021), 124519.
\bibitem{2AH2021}{\sc V. Allu and H. Halder}, Bohr phenomenon for certain subclasses of harmonic mappings, {\it Bull. Sci. Math.} {\bf173} (2021), 103053.
\bibitem{1AH2022}{\sc V. Allu and H. Halder}, Operator valued analogues of multidimensional Bohr's inequality, {\it Canadian Math. Bull.} {\bf65}(4) (2022), 1020--1035.
\bibitem{AA2023} {\sc V. Allu and V. Arora}, Bohr-Rogosinski type inequalities for concave univalent functions, {\it J. Math. Anal. Appl.} {\bf520} (2023), 126845.
\bibitem{B2024} {\sc R. Biswas}, Second Hankel determinant of logarithmic coefficients for $\mathcal{G}(\alpha)$ and $\mathcal{P}(M)$, {\it J. Anal.} (2024). https://doi.org/10.1007/s41478-024-00778-5.
\bibitem{BDK2004} {\sc C. B\'en\'eteau, A. Dahlner and D. Khavinson}, Remarks on the Bohr phenomenon, {\it Comput. Methods Funct. Theory} {\bf4}(1) (2004), 1--19.
\bibitem{BK1997} {\sc H. P. Boas and D. Khavinson}, Bohr's power series theorem in several variables, {\it Proc. Am. Math. Soc.} {\bf125}(10) (1997), 2975--2979.
\bibitem{B1914} {\sc H. Bohr}, A theorem concerning power series, {\it Proc. London Math. Soc.} {\bf13}(2) (1914), 1--5.
\bibitem{BB2004}{\sc E. Bombieri and J. Bourgain}, A remark on Bohr's inequality, {\it Int. Math. Res. Not.} {\bf80} (2004), 4307--4330.
\bibitem{C1940} {\sc F. Carlson}, Sur les coefficients d'une fonction born\'ee dans le cercle unit\'e (French), {\it Ark. Mat. Astr. Fys.} {\bf 27A}(1) (1940), 8 pp.
\bibitem{DP2008}{\sc S. Y.  Dai and Y. F. Pan}, Note on Schwarz-Pick estimates for bounded and positive real part analytic functions, {\it Proc. Am. Math. Soc.} {\bf136} (2008), 635--640.
\bibitem{DFOOS2011} {\sc A. Defant, L. Frerick, J. Ortega-Cerd\`a, M. Ouna\"ies and K. Seip}, The Bohnenblust-Hille inequality for homogeneous polynomials is hypercontractive, {\it Ann. Math.} {\bf174}(2) (2011), 512--517.
\bibitem{DGMS2019} {\sc A. Defant, D. Garc\'ia, , M. Maestre and P. Sevilla-Peris}, Dirichlet series and holomorphic functions in high dimension, {\bf 37}, {\it Cambridge University Press}, 2019.
\bibitem{D1995} {\sc P. G. Dixon}, Banach algebras satisfying the non-unital von Neumann inequality, {\it Bull. Lond. Math. Soc.} {\bf27}(4) (1995), 359--362.
\bibitem{D2004} {\sc P. L. Duren}, Harmonic mapping in the plane, {\it Cambridge University Press}, 2004.
\bibitem{GMR2018} {\sc S. R. Garcia, J. Mashreghi and W. T. Ross}, Finite Blaschke Products and Their Connections, Springer, Cham, 2018.
\bibitem{G2002}{\sc R. M. Goodloe}, Hadamard products of convex harmonic mappings, {\it Complex Var. Theory Appl.} {\bf47} (2002), 81--92. 
\bibitem{IKP2020} {\sc A. Ismagilov, I. R. Kayumov and S. Ponnusamy}, Sharp Bohr type inequality, {\it J. Math. Anal. Appl.} {\bf489} (2020), 124147.
\bibitem{IKKP2021} {\sc A. Ismagilov, A. V. Kayumova, I. R. Kayumov and S. Ponnusamy}, Bohr inequalities in some classes of analytic functions, {\it J. Math. Sci.} {\bf252}(3) (2021), 360--373.
\bibitem{K2008} {\sc D. Kalaj}, Quasiconformal harmonic mapping between Jordan domains, {\it Math. Z.} {\bf260}(2) (2008), 237--252.
\bibitem{KKP2021}  {\sc I. R. Kayumov, D. M. Khammatova} and {\sc S. Ponnusamy}, Bohr-Rogosinski phenomenon for analytic functions and Ces\'aro operators, {\it J. Math. Anal. Appl.} {\bf 493}(2) (2021), 124824. 
\bibitem{KKP2022}  {\sc I. R. Kayumov, D. M. Khammatova} and {\sc S. Ponnusamy}, The Bohr inequality for the generalized Ces\'aro averaging operators, {\it Mediterr. J. Math.} {\bf 19} (2022), 19. 
\bibitem{1KP2017}  {\sc I. R. Kayumov} and {\sc S. Ponnusamy}, Bohr-Rogosinski radius for analytic functions, preprint, https://doi.org/10.48550/arXiv.1708.05585.
\bibitem{1KP2018}{\sc I. R. Kayumov} and {\sc S. Ponnusamy}, Improved version of Bohr's inequality, {\it C. R. Math. Acad. Sci. Paris} {\bf356}(3) (2018), 272--277.
\bibitem{2KP2018}{\sc I. R. Kayumov} and {\sc S. Ponnusamy}, Bohr's inequalities for the analytic functions with lacunary series and harmonic functions, {\it J. Math. Anal. Appl.} {\bf465} (2018), 857--871.
\bibitem{KPS2018}{\sc I. R. Kayumov, S. Ponnusamy} and {\sc N. Shakirov}, Bohr radius for locally univalent harmonic mappings, {\it Math. Nachr.} {\bf291} (2018), 1757--1768.
\bibitem{K2006} {\sc S. G. Krantz}, Geometric Function Theory. Explorations in Complex Analysis. Birkh\"auser, Boston, 2006.
\bibitem{L1936} {\sc H. Lewy}, On the non-vanishing of the Jacobian in certain one-to-one mappings, {\it Bull. Am. Math. Soc.} {\bf42} (1936), 689--692.
\bibitem{LLP2021} {\sc G. Liu, Z. H.  Liu} and {\sc S. Ponnusamy}, Refined Bohr inequality for bounded analytic functions, {\it Bull. Sci. Math.} {\bf173} (2021), 103054.
\bibitem{LP2021} {\sc M. S. Liu} and {\sc S. Ponnusamy}, Multidimensional analogues of refined Bohr's inequality, {\it Proc. Amer. Math. Soc.} {\bf149} (2021), 2133--2146.
\bibitem{LP2023} {\sc G. Liu} and {\sc S. Ponnusamy}, Improved Bohr inequality for harmonic mappings, {\it Math. Nachr.} {\bf296} (2023), 716--731.
\bibitem{LSX2018} {\sc M. S. Liu, Y. M. Shang} and {\sc J. F. Xu}, Bohr-type inequalities of analytic functions, {\it J. Inequal. Appl.} {\bf345} (2018), 13 pp.
\bibitem{M1968} {\sc O. Martio}, On harmonic quasiconformal mappings, {\it Ann. Acad. Sci. Fenn. A. I.} {\bf425} (1968), 3--10.
\bibitem{MBG2024} {\sc R. Mandal, R. Biswas} and {\sc S. K. Guin}, Geometric studies and the Bohr radius for certain normalized harmonic mappings, {\it Bull. Malays. Math. Sci. Soc.} {\bf47} (2024), 131.
\bibitem{MRA2024}{\sc S. Mandal, P. P. Roy} and {\sc M. B. Ahamed}, Hankel and Toeplitz determinants of logarithmic coefficients of inverse functions for certain classes of univalent functions, {\it Iran J. Sci.} (2024).
\bibitem{PS2004} {\sc V. I. Paulsen} and {\sc D. Singh}, Bohr's inequality for uniform algebras, {\it Proc. Amer. Math. Soc.} {\bf132}(12) (2004), 3577--3579.
\bibitem{PVW2020} {\sc S. Ponnusamy, R. Vijayakumar} and {\sc K. -J. Wirths}, New inequalities for the coefficients of unimodular bounded functions, {\it Results Math} {\bf 75} (2020), 107.
\bibitem{PVW2021}{\sc S. Ponnusamy, R. Vijayakumar} and {\sc K. -J. Wirths}, Modifications of Bohr's inequality in various settings, Houston J. M. {\bf47}(4)(2021), 811--835.
\bibitem{QV2005}{\sc S. -L. Qiu} and {\sc M. Vuorinen}, Special functions in geometric function theory, Handbook of Complex Analysis, {\bf2}, North-Holland, 621--659 (2005).
\bibitem{R1923} {\sc W. Rogosinski}, \"Uber Bildschranken bei Potenzreihen und ihren Abschnitten, {\it Math. Z.} {\bf17} (1923), 260--276.
\bibitem{SS1989}{\sc R. Singh} and {\sc S. Singh}, Convolution properties of a class of starlike functions, {\it Proc. Am. Math. Soc.} {\bf106} (1989), 145--152.
\end{thebibliography}
\end{document}